\documentclass[reqno,11pt]{amsart}

\usepackage{amssymb}
\usepackage{amsmath}
\usepackage{amsthm}
\usepackage{amsfonts}
\usepackage{xcolor} 
\usepackage{graphicx}
\usepackage{mathptmx}

\usepackage{enumitem}
\usepackage{comment}
\usepackage{setspace}
\usepackage{wrapfig}

\usepackage[colorlinks=true, pdfstartview=FitV, linkcolor=blue, citecolor=blue, urlcolor=blue]{hyperref}

\usepackage[margin=3cm]{geometry}

\usepackage[mathscr]{eucal}

\allowdisplaybreaks
\numberwithin{equation}{section}

\newtheorem{theorem}{Theorem}[section]
\newtheorem{proposition}[theorem]{Proposition}
\newtheorem{lemma}[theorem]{Lemma}
\newtheorem{corollary}[theorem]{Corollary}
\newtheorem{remark}[theorem]{Remark}

\theoremstyle{definition}

\newcommand{\R}{\mathbb{R}}

\renewcommand{\H}{\mathbb{H}}
\newcommand{\U}{\mathbb{U}}
\newcommand{\V}{\mathbb{V}}
\newcommand{\T}{\mathbb{T}}
\newcommand{\Z}{\mathbb{Z}}

\newcommand{\sH}{\mathscr{H}}

\newcommand{\cA}{\mathcal{A}}
\newcommand{\cB}{\mathcal{B}}
\newcommand{\cE}{\mathcal{E}}

\newcommand{\cP}{\mathcal{P}}

\newcommand{\cl}[1]{\overline{#1}}
\newcommand{\p}{\partial}

\newcommand{\ess}{{\mathrm ess}}
\newcommand{\disc}{{\mathrm disc}}
\newcommand{\diag}{\mathrm {diag}}

\renewcommand{\tilde}{\widetilde}
\renewcommand{\hat}{\widehat}
\newcommand{\wk}{\rightharpoonup}
\newcommand{\sign}{\mathrm{sign}\,}

\setlist[itemize]{leftmargin=6mm} 

\title{Bound states of $2+1$ fermionic trimers on lattice at strong couplings}

\author[J. Abdullaev]{Janikul Abdullaev}
\address[J. Abdullaev]{Samarkand State University named after Sharof Rashidov, University Boulevard 15, Samarkand 140104, Uzbekistan}
\email{jabdullaev@mail.ru}

\author[A. Khalkhuzhaev]{Ahmad Khalkhuzhaev}
\address[A. Khalkhuzhaev]{Samarkand State University named after Sharof Rashidov, University Boulevard 15, Samarkand 140104, Uzbekistan}
\email{ahmad\_x@mail.ru}

\author[Sh. Kholmatov]{Shokhrukh Yu. Kholmatov} 
\address[Sh. Kholmatov]{University of Vienna, Oskar-Morgenstern Platz 1, 1090 Vienna, Austria \& Samarkand State University named after Sharof Rashidov, University Boulevard 15, Samarkand 140104, Uzbekistan}
\email{shokhrukh.kholmatov@univie.ac.at}



\date{\today}
\begin{document}

\begin{abstract}
In this paper, we investigate the bound states of $2+1$ fermionic trimers on a three-dimensional lattice at strong coupling. Specifically, we analyze the discrete spectrum of the associated three-body discrete Schr\"odinger operator $H_{\gamma,\lambda}(K),$ focusing on energies below the continuum and within its gap. Depending on the quasi-momentum $K,$ we show that if the mass ratio $\gamma>0$ between the identical fermions and the third particle is below a certain threshold, the operator lacks a discrete spectrum below the essential spectrum for sufficiently large coupling $\lambda>0.$ Conversely, if $\gamma$ exceeds this threshold, $H_{\gamma,\lambda}(K)$ admits at least one eigenvalue below the essential spectrum. Similar phenomena are observed in the neighborhood of the two-particle branch of the essential spectrum, which resides within the gap and grows sublinearly as $\lambda\to+\infty.$ For $K=0,$ the mass ratio thresholds are explicitly calculated and it turns out that, for certain intermediate mass ratios and large couplings, bound states emerge within the gap, although ground states are absent.

\end{abstract}

\maketitle

\section{Introduction and main results}

\subsection{Background}
The discrete Schr\"odinger operators model various phenomena ranging from crystal lattice vibrations to tight-binding approximations in quantum mechanics. While the spectral properties of single-particle and two-body lattice systems are now sufficiently well-understood (see e.g. \cite{ALMM:2006,KhLA:2022,LKhKh:2021,Mattis:1986} and the references therein), not much seems known on the discrete spectrum of Hamiltonians associated to three or more particle systems with short-range pair interactions even in the continuous case. For instance, some sufficient conditions for existence of bound states are obtained via variational principles in \cite{BChK:1986,ChM:1980,KhM:2018,LDaKh:2016,PMC:1985}; see also \cite{BMO:2018,HS:2000,KSII:1980,Sigal:1976}. The most prominent example of existence is the Efimov effect \cite{ALM:2004,BT:2017,FFT:2023,Gridnev:2013,Gridnev:2014,Lakaev:1993,Muminov:2009,NE:2017,Sobolev:1993,Tamura:1991,Yafaev:1974}.
The main difficulties in the study usually arise from the interplay of combinatorial complexity, long-range interactions, and the potential emergence of novel spectral phenomena. 

\subsection{The model} 
In this work we consider the three-dimensional quantum system in the three-dimensional optical lattice $\Z^3,$ consisting of two identical fermions and a third particle of different species and mass ratio $\gamma$ with respect to the mass of the fermion, where the inter-particle interaction is of zero range. As in \cite{BMO:2018}, for shortness, we refer to this system as the $2+1$ fermionic trimer with zero-range interaction; note that by the exclusion principle, the interaction is only present between each fermion and the third particle.

In the coordinate representation the  corresponding Hamiltonian $\hat\H_{\lambda,\gamma}$ is defined in the subspace 
$$
\hat \sH:=\{\hat f\in \ell^{2}((\Z^3)^3):\, \hat f(x^1,x^2,x^3) =-\hat f(x^2,x^1,x^3),\quad x^1,x^2,x^3\in\Z^3\}
$$ 
of the Hilbert space $\ell^{2}((\Z^3)^3)$ as 
$$
\hat \H_{\gamma,\lambda} = \hat \H_{\gamma,0} - \lambda \hat \V,\quad \lambda\ge0.
$$
Here the free Hamiltonian is the sum of three discrete Laplacians
$$
\hat \H_{\gamma,0} = \hat \Delta \otimes \hat I \otimes \hat I + \hat I \otimes \hat \Delta \otimes \hat I + \gamma\, \hat I  \otimes \hat I \otimes \hat \Delta,
$$
where after rescaling the fermions are assumed having unit mass, and third different particle has mass $1/\gamma>0,$
$$
\hat \Delta \hat f(x) = \frac12\sum_{|s|=1} [\hat f(x) - \hat f(x+s)].
$$
The potential operator $\hat \V= \hat \V_{13}+\hat \V_{23}$ is defined as 
$$
\hat \V_{\alpha3} f(x^1,x^2,x^3) = 
\begin{cases}
f(x^1,x^2,x^3) & \text{if $x^\alpha=x^3,$}\\
0 & \text{otherwise,}
\end{cases}
\quad \alpha=1,2.
$$
As mentioned earlier, by the exclusion principle, the zero-range interaction $\hat \V_{12}$ between fermions is zero. Since $\hat\H_{\gamma,\lambda}$ commutes with members of one-parameter family of translation operators $\{\U_s\}_{s\in\Z^3},$ given by $\U_s\hat f(x_1,x_2,x_3)=f(x_1+s,x_2+s,x_3+s),$ we can decompose the Hamiltonian $\hat\H_{\gamma,\lambda}$ and the underlying space $\hat\sH$ into von Neumann direct integrals. In the coordinate system, relative to the motion of the third particle, we have 
$$
\hat\sH \cong \int_{K\in\T^3}^\oplus L^{2,a}((\T^3)^2)dK,\quad  \hat \H_{\gamma,\lambda}\cong \int_{K\in\T^3}^\oplus H_{\gamma,\lambda}(K)\,dK,
$$
where $\T^3:=(-\pi,\pi]^3$ is the three dimensional torus  (Brillouin zone), $L^{2,a}((\T^3)^2)$ is the  Hilbert space of square-integrable antisymmetric function in $\T^3 \times\T^3,$ and the fiber operators 
$
H_{\gamma,\lambda}(K)
$
are given by 
$$
H_{\gamma,\lambda}(K) = H_{\gamma,0}(K) - \lambda (V_{13}+V_{23}).
$$
Here the free operator $H_{\gamma,0}(K)$ is multiplication by the function 
$$
\cE_{\gamma,K}(p,q) = \epsilon(p)+\epsilon(q)+\epsilon(K-p-q),\quad p,q\in\T^3,
$$
with 
$$
\epsilon(p):=\sum_{i=1}^3(1-\cos p_i),
$$
and $V_{\alpha3}$ are the (relative) interaction potentials
\begin{align*}
V_{13}f(p,q) = &(2\pi)^{-3}\int_{\T^3}f(s,q)ds,\\ 
V_{23}f(p,q) = & (2\pi)^{-3}\int_{\T^3}f(p,s)ds.
\end{align*}
For more details of such a decomposition, we refer, for instance, to \cite{ALM:2004,KhM:2018} and the references therein.

Notice that using a change of variables, by means of the unitary operator $Uf(p,q) = f(-p,-q),$  we can show that $H_{\gamma,\lambda}(K)$ and $H_{\gamma,\lambda}(-K)$ are unitarily equivalent for any $K\in\T^3.$

The essential spectrum of $H_{\gamma,\lambda}(K)$ consists of a union of two segments 
\begin{equation*}
\sigma_{\ess}(H_{\gamma,\lambda}(K))=[\tau_{\gamma,\min }(K,\lambda),
\tau_{\gamma,\max }(K,\lambda)]\cup [\cE_{\gamma,\min}(K), \cE_{\gamma,\max}(K)],
\end{equation*}
where the segments $[\tau_{\gamma,\min }(K,\lambda),
\tau_{\gamma,\max }(K,\lambda)]$ and $[\cE_{\gamma,\min}(K), \cE_{\gamma,\max}(K)]$ are the \emph{two-particle}  and the \emph{three-particle} branches of the essential spectrum of $H_{\gamma,\lambda}(K),$ respectively, 
\begin{equation*}
\cE_{\gamma,\min}(K) = \min_{p,q}\,\cE_{\gamma,K}(p,q),\quad \cE_{\gamma,\max}(K) = \max_{p,q}\,\cE_{\gamma,K}(p,q),
\end{equation*}
and 
\begin{equation}\label{tau_estotot}
\cE_{\gamma,\min}(K) - \lambda \le \tau_{\gamma,\min }(K,\lambda) \le 
\tau_{\gamma,\max }(K,\lambda)\le \cE_{\gamma,\max}(K) - \lambda,
\end{equation}
see e.g. relations  \eqref{minmax_z_gam_lam},  \eqref{zu712s} and \eqref{shdzu67ec} below. Thus, for large $\lambda,$ the branches of the essential are disjoint. In particular, a discrete spectrum could appear in the gap between these branches.

\begin{remark}\label{rem:varia_not_enaf}
As the pair-potentials $V_{\alpha3}$ are self-adjoint and $(V_{13}+V_{23})^2=V_{13}+V_{23},$ the perturbation $V_{13}+V_{23}$ of $H_{\gamma,\lambda}(K)$ is an orthogonal projection. In particular,  $\|V_{13}+V_{23}\|=1,$ and hence, by the  minmax principle 
$$
\inf\sigma(H_{\gamma,\lambda}(K)) = \inf_{f\in L^{2,a}((\T^3)^2),\,\|f\|_{L^2}=1} (H_{\gamma,\lambda}(K)f,f)_{L^2}
$$
we get 
\begin{equation}\label{bound_sigma_inf12}
\cE_{\gamma,\min}(K) - \lambda \le \inf\sigma(H_{\gamma,\lambda}(K)) \le \cE_{\gamma,\max}(K) - \lambda.
\end{equation}
From the bounds \eqref{bound_sigma_inf12} and \eqref{tau_estotot} it follows that both $\inf\sigma(H_{\gamma,\lambda}(K))$ and $\inf\sigma_\ess(H_{\gamma,\lambda}(K))$ have the same asymptotics for large $\lambda.$ Thus, the variational methods does not provide enough information related to the existence of the discrete spectrum of $H_{\gamma,\lambda}(K).$
\end{remark}

\subsection{Main results}

We are interested in the discrete spectrum of $H_{\gamma,\lambda}(K)$ for the given mass ratio $\gamma>0,$ quasimomentum $K\in\T^3$ and large $\lambda>0$ (depending basically only on $\gamma$).

We first study the eigenvalues lying below the essential spectrum.

\begin{theorem}[Eigenvalues of $H_{\gamma,\lambda}(K)$ below the essential spectrum]\label{teo:exist_eigen_H}
Let $K\in\T^3.$ There exist $\gamma_1:=\gamma_1(K)>0$ and $\tilde \gamma_1:=\tilde \gamma_1(K)>0$ with the following property.

\begin{itemize}
\item For any $\gamma\in(0,\gamma_1)$ there exists $\lambda_1:=\lambda_1(K,\gamma)>0$ such that for any $\lambda>\lambda_1$ the operator $H_{\gamma,\lambda}(K)$ has no eigenvalues below its essential spectrum.

\item For any $\gamma>\gamma_1$ there exists $\lambda_2:=\lambda_2(K,\gamma)>0$ such that for any $\lambda>\lambda_2$ the operator $H_{\gamma,\lambda}(K)$ has an eigenvalue below its essential spectrum.

\item For any $\gamma>\tilde \gamma_1$ there exists $\lambda_3:=\lambda_3(K,\gamma)>0$ such that for any $\lambda>\lambda_3$ the operator $H_{\gamma,\lambda}(K)$ has at least three eigenvalues below its essential spectrum.

\end{itemize}

\end{theorem}


Next we study the eigenvalues in the gap, but close to the two-particle branch of the essential spectrum.

\begin{theorem}[Eigenvalues of $H_{\gamma,\lambda}(K)$ in the gap]\label{teo:exist_eigen_H_gap}
Let $K\in\T^3$ and let $\lambda\mapsto T_\lambda$ be any strictly increasing function in $(0,+\infty)$ with 
$$
\lim_{\lambda\to+\infty} \frac{\lambda}{T_\lambda}=+\infty.
$$
There exist $\gamma_2:=\gamma_2(K)>0$ and $\tilde \gamma_2:=\tilde \gamma_2(K)>\gamma_2$ with the following property.

\begin{itemize}
\item For any $\gamma\in(0,\gamma_2)$ there exists $\lambda_4:=\lambda_4(K,\gamma)>0$ such that for any $\lambda>\lambda_4$ the operator $H_{\gamma,\lambda}(K)$ has no eigenvalues in the gap $[\tau_{\gamma,\max}(K,\lambda), \tau_{\gamma,\max}(K,\lambda) + T_\lambda];$

\item For any $\gamma>\gamma_2$ there exists $\lambda_5:=\lambda_5(K,\gamma)>0$ such that for any $\lambda>\lambda_5$ the operator $H_{\gamma,\lambda}(K)$ has an eigenvalue in the gap $[\tau_{\gamma,\max}(K,\lambda), \tau_{\gamma,\max}(K,\lambda) + T_\lambda].$

\item For any $\gamma>\tilde \gamma_2$ there exists $\lambda_6:=\lambda_6(K,\gamma)>0$ such that for any $\lambda>\lambda_6$ the operator $H_{\gamma,\lambda}(K)$ has at least three eigenvalues in the gap $[\tau_{\gamma,\max}(K,\lambda), \tau_{\gamma,\max}(K,\lambda) + T_\lambda].$

\end{itemize}
\end{theorem}

\begin{corollary}[Satellite eigenvalues of two particle branch of essential spectrum]
Let $K\in\T^3.$ If $\gamma>\max\{\tilde \gamma_1(K), \tilde \gamma_2(K)\},$ then for large $\lambda>0$ the operator $H_{\gamma,\lambda}(K)$ has at least three eigenvalues below  and at least three eigenvalues  above the two particle branch of its essential spectrum. 
\end{corollary}

Unfortunately, due to the nontrivial dependence of $H_{\gamma,\lambda}(K)$ on $K,$ our methods provide limited results in the general case. However, for the special case $K=0,$ we can derive additional insights into the discrete spectrum, addressing both quantitative and qualitative aspects.

The following theorem improves Theorem \ref{teo:exist_eigen_H} not only presenting the existence or absence of discrete spectrum $H_{\gamma,\lambda}(0),$  but also indicating the exact number of eigenvalues.

\begin{theorem}
\label{teo:eiegn_H_below} 
Let $\gamma_1(K),$ $\lambda_1(K,\gamma),$ $\lambda_2(K,\gamma)$  and $\lambda_3(K,\gamma)$ be given by Theorem \ref{teo:exist_eigen_H}. Then  
\begin{equation}\label{def:gamma_1}
\gamma_1(0)=\tilde \gamma_1(0)=\Big(\frac{1}{ (2\pi)^3}\int_{\T^3} \frac{\sin^2p_1 dp}{\epsilon(p)}\Big)^{-1}\approx 4.7655 
\end{equation}
and $\lambda_2(0,\gamma) = \lambda_3(0,\gamma).$ Thus:
\begin{itemize}
\item[\rm(a)] if $\gamma\in(0,\gamma_1(0)),$ then for any $\lambda>\lambda_1(0,\gamma)$ the operator $H_{\gamma,\lambda}(0)$ has no discrete spectrum below its essential spectrum.

\item[\rm(b)] if $\gamma >\gamma_1(0),$ then for any $\lambda>\lambda_2(0,\gamma)$ the operator $H_{\gamma,\lambda}(0)$ has exactly one eigenvalue $z$ lying below its essential spectrum. Moreover, $z$ is a triple eigenvalue (i.e., its multiplicity is three). Moreover, the associated eigenfunctions are odd in $\T^3\times\T^3.$
\end{itemize}
\end{theorem}

\noindent 
Some explicit choices of $\lambda_1(0,\gamma)$ and $\lambda_2(0,\gamma)$ will be given in  the proof.

Now we consider the eigenvalues of $H_{\gamma,\lambda}(0)$ in the gap.

\begin{theorem}[Eigenvalues in the gap]\label{teo:eiegn_H_gap}
Let $\gamma_2(K),$ $\tilde\gamma_2(K),$ $\lambda_4(K,\gamma),$ $\lambda_5(K,\gamma),$ $\lambda_6(K,\gamma)$ and $T_\lambda$ be as in Theorem \ref{teo:exist_eigen_H_gap}. Then
$$
\tilde \gamma_2(0):=\Big(\frac{1}{2\cdot(2\pi)^3}\int_{\T^3} 
\frac{(\cos p_1-\cos p_2)^2dp}{\epsilon(\vec\pi-p)} \Big)^{-1}\approx 5.398489,
$$
and $0<\gamma_2(0)\le 2.93652<\tilde\gamma_2(0).$ Moreover,
\begin{itemize}
\item[\rm(a)] for any $\gamma<\gamma_2(0)$ and $\lambda>\lambda_4(0,\gamma)$ the operator $H_{\gamma,\lambda}(0)$ has no eigenvalues in the portion $[\tau_{\gamma,\max}(0,\lambda),\tau_{\gamma,\max}(0,\lambda)+T_\lambda]$ of the gap;

\item[\rm(b)] for any $\gamma\in(\gamma_2(0),\tilde \gamma_2(0))$ and $\lambda>\lambda_5(0,\gamma)$ the operator $H_{\gamma,\lambda}(0)$ has at least one eigenvalue in $[\tau_{\gamma,\max}(0,\lambda),\tau_{\gamma,\max}(0,\lambda) + T_\lambda];$

\item[\rm(c)] for any $\gamma>\tilde \gamma_2(0)$ and $\lambda>\lambda_6(0,\gamma)$ the operator $H_{\gamma,\lambda}(0)$ has at least three eigenvalues in $[\tau_{\gamma,\max}(0,\lambda) ,\tau_{\gamma,\max}(0,\lambda)+T_\lambda].$
\end{itemize}
Finally, any eigenfunction corresponding to an eigenvalue $z\in (\tau_{\gamma,\max}(0,\lambda), \tau_{\gamma,\max}(0,\lambda) + T_\lambda)$ is an even function in $\T^3\times\T^3.$ Moreover, all eigenvalues given by (b) and (c) are in fact belong to the intervals $(\tau_{\gamma,\max}(0,\lambda),\tau_{\gamma,\max}(0,\lambda) +3.5\gamma)$ and $(\tau_{\gamma,\max}(0,\lambda),\tau_{\gamma,\max}(0,\lambda) + \gamma),$ respectively.
\end{theorem}

\noindent 
Since $\gamma_2(0)<\gamma_1(0),$ bound states exist within the gap; however, there are no ground states.

\subsection{Some comments on the main results}

According to Theorems \ref{teo:exist_eigen_H}-\ref{teo:eiegn_H_gap} for small $\gamma$ and large couplings the associated $2+1$ fermionic system does not admit bound states. In other words, at strong couplings $2+1$ fermionic trimers exist if and only if the mass of the fermion is large enough than the mass of the third particle.

\subsection{A main strategy of proofs} 

The main tool in the analysis of the discrete spectrum of a Birman-Schwinger-type operator $\cB_{\gamma,K}(z,\lambda),$ which will be introduced in Section \ref{sec:bsh_operator_definos}. Namely, we reduce the eigenvalue problem for $H_{\gamma,\lambda}(K)$ to an eigenvalue problem for some compact operator. Similar method is well-known in the spectral theory of one-particle operators $\Delta - \lambda V:$ in this setting the associated compact operator is $V^{1/2}(\Delta -zI)^{-1}V^{1/2}$ (under assumption $V\ge0$). 

Unlike few-body case, the operator $\cB_{\gamma,K}(z,\lambda)$ is very complex and a priori we have no monotonicity in $z$ or $\lambda.$ However, at strong coupling (i.e. large $\lambda$) the behaviour of this operator can be handled: namely we can represent it as a sum of a finite rank integral operator (a principal part) and an compact operator (a residual part) with small norm (for large $\lambda$ and $|z|$). We can establish this by thoroughly analyzing  the eigenvalues of the Hamiltonian, associated to the system  of one fermion of third particle, which is also important in the study of HVZ theorem for $H_{\gamma,\lambda}(K),$ see Section \ref{sec:essential_spec}. 

In view of \eqref{tau_estotot} the largeness of $|z|$ is trivially obtained in the case $z$ below the essential spectrum. 
However, in the gap, the largeness of  $\lambda$ does not necessarily yield the largeness of $|z|,$ and hence, to extract the  residual part we need to assume some growth condition on it. One sufficient condition can be obtaining assuming $z$ diverges sublinearly to $-\infty$  (see assumption of Theorems \ref{teo:exist_eigen_H_gap} and \ref{teo:eiegn_H_gap}). However, we do not state existence or nonexistence of discrete spectrum near the three-particle branch of the essential spectrum.

In case $K=0,$ we can explicitly compute the invariant subspaces of the operator $\cB_{\gamma,0}(z,\lambda)$ and its principal parts, see Section \ref{sec:bsh_oper_for_large_lambda}.
Namely, it turns out that the subspaces of even and odd functions in $L^2(\T^3)$ are invariant subspaces of  $\cB_{\gamma,0}(z,\lambda),$ 
$$
\cB_{\gamma,0}(z,\lambda)\big|_{L^{2,e}(\T^3)}\le0,\qquad 
\cB_{\gamma,0}(z,\lambda)\big|_{L^{2,o}(\T^3)}\ge0\quad\text{if $z<\tau_{\gamma,\min}(0,\lambda)$}
$$
and 
$$
\cB_{\gamma,0}(z,\lambda)\big|_{L^{2,e}(\T^3)}\ge0,\qquad 
\cB_{\gamma,0}(z,\lambda)\big|_{L^{2,o}(\T^3)}\le0\quad\text{if $\tau_{\gamma,\max}(0,\lambda)<z<\cE_{\gamma,\min}(K).$}
$$
This allows to obtain some evenness or oddness properties of eigenfunctions of $\cB_{\gamma,0}(z,\lambda),$ and hence, that of $H_{\gamma,\lambda}(0).$ 
These subspaces are also invariant with respect to the principal parts $\cB_\gamma^p(z,\lambda)$ of $\cB_{\gamma,0}(z,\lambda).$ Now we will treat the cases if gap and below the essential spectrum separately. 
\begin{itemize}
\item If $z<\tau_{\gamma,\min}(0,\lambda),$ then $\cB_\gamma^p(z,\lambda)$ has a single positive eigenvalue of multiplicity three. Moreover, the eigenvalue is strictly increasing in $z,$ and for $z=\tau_{\gamma,\min}(0,\lambda),$ converges to $\gamma/\gamma_1(0)$ as $\lambda\to+\infty.$ 

\item If $z\in (\tau_{\gamma,\max}(0,\lambda),\tau_{\gamma,\max}(0,\lambda) + T_\lambda)$ for some function $\lambda\mapsto T_\lambda$ with sublinear growth, then $\cB_\gamma^p(z,\lambda)$ have two different eigenvalues one simple and one with multiplicity two; the multliplicity two eigenvalue is monotone in $z,$ and for $z=\tau_{\gamma,\max}(0,\lambda),$ as $\lambda\to+\infty$ converges to $\gamma/\tilde\gamma_2(0).$ We do not expect such a monotonicty for the  other eigenvalue, but still we show that its maximum in $z$ in the limit as $\lambda\to+\infty$ has asymptotics equal to $\gamma/\gamma_2(0).$ 
\end{itemize}

Now using the continuity of eigenvalues of $\cB_{\gamma,0}(z,\lambda)$ we construct $z$'s for which $\cB_{\gamma,0}(z,\lambda)$ either admits $1$ as an eigenvalue or has all eigenvalues less than $1.$ Then applying the Birman-Schwinger-type principle (Theorem \ref{teo:bsh_principle}) we deduce the assertions. We mention here that the lack of the monotonicity yield that while $\cB_{\gamma,0}(z,\lambda)$ for $z$ in gap has three eigenvalues greater than $1,$ the operator $H_{\gamma,\lambda}(0)$ has at least three eigenvalues in the gap and may have more than three eigenvalues there.

\subsection{The structure of the paper} 

The paper is organized as follows. In Section \ref{sec:essential_spec} we study the essential spectrum of $H_{\gamma,\lambda}(K)$ and the asymptotics of the unique eigenvalue of the effective discrete Schr\"odinger operators $h_{\gamma,\lambda}(K,q),$ associated to the system of one fermion and the third particle. The Birman-Schwinger-type operator $\cB_{\gamma,K}(z,\lambda)$ and its relation to $H_{\gamma,\lambda}(K)$ will be studied in Section \ref{sec:bsh_operator_definos}. The behaviour of $\cB_{\gamma,K}(z,\lambda)$ for large $\lambda$ and the proofs of the main results is included in Section \ref{sec:bsh_oper_for_large_lambda}. We conclude the paper with an appendix, where we provide a full proof of the HVZ theorem for $H_{\gamma,\lambda}(K)$ (Theorem \ref{teo:ess_spectrum_3}).

\subsection*{Acknowledgements}
The authors thank Prof. Robert Seiringer for his valuable suggestions and insightful comments. Sh. Kholmatov acknowledges
support from the Austrian Science Fund (FWF) Stand-Alone project P 33716.

\section{Essential spectrum. Spectral properties of two-particle operators}\label{sec:essential_spec}

\subsection{Channel operators}

Let us consider the channel operator 
\begin{equation}\label{channel_operator}
H_{\gamma,\lambda}^{(2)}(K):=H_{\gamma,0}(K) - \lambda V_{13}
\end{equation}
in $L^{2}((\T^3)^2),$ associated to the two-particle system, consisting of one fermion and one different particle. Since the summands of $H_{\gamma,\lambda}^{(2)}(K)$ commute with each member of the Abelian group $\{U_s\}_{s\in\Z^3}$ of multiplication operators in $L^2((\T^3)^2)$ by $q\mapsto e^{is\cdot q},$ both $H_{\gamma,\lambda}^{(2)}(K)$ and $L^2((\T^3)^2)$ are decomposed into the direct von Neumann integral 
\begin{equation}\label{decompos1e}
L^2((\T^3)^2) = \int_{q\in\T^3}^\oplus L^2(\T^3)\,dq, \quad H_{\gamma,\lambda}^{(2)} (K) = \int_{q\in\T^3}^\oplus h_{\gamma,\lambda}(K,q)\,dq,
\end{equation}
where in the two-particle fibers 
$$
h_{\gamma,\lambda}(K,q): = h_{\gamma,0}(K,q) - \lambda v
$$ 
the free operator $h_{\gamma,0}(K,q)$ is the multiplication in $L^2(\T^3)$ by $\cE_{\gamma,K}(\cdot,q)$ and $v$ is the rank-one integral operator, defined as 
$$
vf(p) = \frac{1}{(2\pi)^{3}}\int_{\T^3} f(s)\,ds,\quad f\in  L^2(\T^3).
$$
Since $q\mapsto h_{\gamma,\lambda}(K,q)$ is analytic family,
\begin{equation}\label{hsz172}
\sigma(H_{\gamma,\lambda}^{(2)}(K)) = \sigma_\ess(H_{\gamma,\lambda}^{(2)}(K)) = \bigcup_{q\in\T^3}\,\sigma(h_{\gamma,\lambda}(K,q)).
\end{equation}
Moreover, by the classical Weyl's theorem, 
$$
\sigma_\ess(h_{\gamma,\lambda}(K,q)) = \sigma(h_{\gamma,0}(K,q)) = [\cE_{\gamma,\min}(K,q),\cE_{\gamma,\max}(K,q)]\quad\text{for all $\lambda\ge0.$}
$$
Here 
\begin{equation*}
\cE_{\gamma,\min}(K,q) = \min_p\,\cE_{\gamma,K}(p,q),\quad \cE_{\gamma,\max}(K,q) = \max_p\,\cE_{\gamma,K}(p,q).
\end{equation*}
Moreover, as $h_{\gamma,\lambda}(K,q)$ is a nonpositive rank-one perturbation of $h_{\gamma,0}(K,q),$ it admits at most one eigenvalue $z_{\gamma,\lambda}(K,q)$  outside the essential spectrum, which lies below the essential spectrum (if exists). Moreover, $z_{\gamma,\lambda}(K,q)$ is the unique zero of the Fredholm determinant
\begin{equation}\label{fredholm_determ}
\Delta_{\gamma,K}(q,z,\lambda):= 1 - \frac{\lambda}{(2\pi)^3}\int_{\T^3}\frac{dp}{\cE_{\gamma,K}(p,q) - z},\quad z\le \cE_{\gamma,\min}(K,q),
\end{equation}
which exists if and only if  
\begin{equation}\label{ahszeur6}
\lambda>\Big(\frac{1}{(2\pi)^3}\int_{\T^3} \frac{dp}{\cE_{\gamma,K}(p,q) - \cE_{\gamma,\min}(K,q)}\Big)^{-1}\in [0,+\infty).
\end{equation}
Note that 
$$
\bigcup_{q\in\T^3}\,\sigma(h_{\gamma,0}(K,q)) = [\cE_{\gamma,\min}(K),\cE_{\gamma,\max}(K)].
$$
Thus, \eqref{hsz172} is represented as
\begin{equation}\label{shdzu67ec}
\sigma(H_{\gamma,\lambda}^{(2)}(K)) = \sigma_\ess(H_{\gamma,\lambda}^{(2)}(K)) = [\cE_{\gamma,\min}(K),\cE_{\gamma,\max}(K)]\cup \bigcup_{q\in\T^3}\,\{z_{\gamma,\lambda}(K,q)\},
\end{equation}
where we set $z_{\gamma,\lambda}(K,q):=\cE_{\gamma,\min}(K,q)$ if $\lambda$ does not satisfy \eqref{ahszeur6}.

\begin{remark}\label{rem:min_max_principle}
Since $z_{\gamma,\lambda}(K,q) $ is the lower bound of $h_{\gamma,\lambda}(K,q),$ by the minmax principle  
\begin{equation}\label{lower_bound_hKk}
z_{\gamma,\lambda}(K,q):= \inf_{\|f\|_{L^2}=1} (h_{\gamma,\lambda}(K,q)f,f)_{L^2} = 
\inf_{\|f\|_{L^2}=1} \Big((h_{\gamma,0}(K,q)f,f)_{L^2} - \lambda(vf,f)_{L^2}\Big).
\end{equation}
In particular, using $\|v\| = \sup_{\|f\|=1} (vf,f)_{L^2} = 1$ we find 
\begin{equation}\label{minmax_z_gam_lam}
\cE_{\gamma,\min}(K,q) - \lambda \le z_{\gamma,\lambda}(K,q) \le \cE_{\gamma,\max}(K,q) - \lambda. 
\end{equation}
\end{remark}

\subsection{Some comments on the spectrum of $h_{\gamma,\lambda}(K,q)$}

In this section we study some important asymptotics of the eigenvalues of $h_{\gamma,\lambda}(K,q).$ 

For $\gamma,\lambda>0$ and  $q\in\T^3$ let 
$$
\hat h_{\gamma,\lambda}(q)f(p):=
\Big(\epsilon(p) + \gamma \epsilon(q-p)\Big)f(p) - \frac{\lambda}{(2\pi)^3}\int_{\T^3} f(s)ds,\quad f\in L^2(\T^3).
$$
Note that 
\begin{equation}\label{ersiaosdx}
h_{\gamma,\lambda}(K,q) = \hat h_{\gamma,\lambda}(K-q) + \epsilon(q)I.
\end{equation}
As 
\begin{align*}
\hat\epsilon(p,q):=\epsilon(p) + \gamma \epsilon(q-p) = & \sum_{i=1}^3 \Big(1 + \gamma - \cos p_i-\gamma\cos(q_i-p_i)\Big)\\
= & \sum_{i=1}^3 \Big(1+\gamma - \sqrt{1+\gamma^2+2\gamma\cos q_i}\cos(p_i-\zeta(q_i))\Big),
\end{align*}
where for the moment we omit the dependence on $\gamma$ and  $\zeta(p_i)$ is such that 
\begin{align*}
\sqrt{1+\gamma^2+2\gamma\cos q_i} \cos \zeta(q_i) = 1+\gamma \cos q_i, \quad 
\sqrt{1+\gamma^2+2\gamma\cos q_i} \sin \zeta(q_i) = \gamma \sin q_i.
\end{align*}
Thus, by the Weyl's theorem 
$$
\sigma_\ess(\hat h_{\gamma,\lambda}(q)) = [\hat\epsilon_{\min}(q),\hat\epsilon_{\max}(q)],
$$
where 
\begin{align*}
\hat\epsilon_{\min}(q) = & \min_p\hat\epsilon(p,q) = \sum_{i=1}^3 (1 + \gamma - \sqrt{1+\gamma^2+2\gamma\cos  q_i} ), \\
\hat\epsilon_{\max}(q) = & \max_p\hat\epsilon(p,q) = \sum_{i=1}^3 (1 + \gamma + \sqrt{1+\gamma^2+2\gamma\cos  q_i} ), 
\end{align*}
and by the min-max principle, there is at most one eigenvalue $\hat z_{\gamma,\lambda}(q)<\hat\epsilon_{\min}(q).$  We set $\hat z_{\gamma,\lambda}(q)=\hat\epsilon_{\min}(q)$ if the eigenvalue does not exist so that 
\begin{equation}\label{shri_lanka_gutersh1}
\hat z_{\gamma,\lambda}(q) = \inf\,\sigma(\hat h_{\gamma,\lambda}(q)).
\end{equation}
Moreover, if $\hat z_{\gamma,\lambda}(q)<\hat\epsilon_{\min}(q),$ then $z=\hat z_{\gamma,\lambda}(q)$ is the unique zero of the Fredholm determinant 
\begin{align*}
\hat \Delta(q,z):= & 
1-\frac{\lambda}{(2\pi)^3} \int_{\T^3} \frac{dq}{\hat \epsilon(p,q)-z} \\
= & 1 - \frac{\lambda}{(2\pi)^3} \int_{\T^3} \frac{dp}{\sum_{i=1}^3 (1 + \gamma - \sqrt{1+\gamma^2+2\gamma\cos q_i} \cos p_i) -  z},\quad  z\le \hat\epsilon_{\min}(q).
\end{align*}

\begin{lemma}[Properties of $\hat z_{\gamma,\lambda}$]\label{lem:prop_z_gamlam}
For $\gamma,\lambda>0$ and $q\in\T^3$ let $\hat z_{\gamma,\lambda}(q)$ be given by \eqref{shri_lanka_gutersh1}. 
Then 
\begin{equation}\label{zgamlam_est_0pie}
-\lambda \le \hat z_{\gamma,\lambda}(0) 
\le \hat z_{\gamma,\lambda}(q) \le \hat z_{\gamma,\lambda}(\vec\pi) \le -\lambda + 3(1+\gamma+|1-\gamma|).
\end{equation}
%
Moreover, if $\lambda > 3(1+\gamma),$ then 
\begin{equation}\label{nice_estoa0s}
-\lambda + 3(1+\gamma) 
- \frac{9(1+\gamma)^2}{\lambda} < \hat z_{\gamma,\lambda}(q) <  -\lambda + 3(1+\gamma) 
,\quad q\in\T^3.
\end{equation}
\end{lemma}

\begin{proof}
For shortness, let us omit the dependence of $\hat z_{\gamma,\lambda}$ on $\gamma,\lambda.$ 
As the maps $q\mapsto \hat\Delta(q,z)$ and $q\mapsto \hat\epsilon_{\min}(q)$ are symmetric with respect
to coordinate permutations of $q=(q_1,q_2,q_3)$ and the maps $q_i\mapsto \hat\Delta(q,z)$ and $q_i\mapsto \hat\epsilon_{\min}(q)$ are even, by definition of $\hat z,$ the map  $q\mapsto \hat z(q)$ is symmetric and $q_i\mapsto \hat z(q)$ is even. 

We claim that for each $i=1,2,3,$ the map  $q_i\mapsto \hat z(q)$ is increasing in $[0,\pi].$ Indeed, for simplicity, assume that $i=1.$ Clearly, $q_1\mapsto \hat\epsilon_{\min}(q)$ is strictly increasing. Moreover, let us show that $q_1\mapsto \hat \Delta(q,z)$ strictly increases in $[0,\pi].$  Write the integral in $p_1$ as
\begin{multline*}
\int_{\T}\frac{dp_1}{A - b\cos p_1} = 2\int_0^\pi \frac{dp_1}{A - b\cos p_1} \\
= 2\int_0^{\pi/2} \frac{dp_1}{A - b\cos p_1} + 2\int_0^{\pi/2}\frac{dp_1}{A + b\cos p_1} = \int_0^{\pi/2} \frac{4A\,dp_1}{A^2-b^2\cos^2 p_1},
\end{multline*}
where the integral over $[\pi/2,\pi]$ is reduced to $[0,\pi/2]$ via the change of variables $p_1\mapsto \pi-p_1,$ 
$$
A:=(1+\gamma) + \sum_{i=2}^3 (1 + \gamma - \sqrt{1+\gamma^2+2\gamma\cos q_i} \cos p_i) - z
$$
and 
$$
b:=\sqrt{1+\gamma^2+2\gamma\cos q_1} .
$$
Since $A$ is independent of $q_1$ and $b,$ as a function of $q_1,$ strictly decreases in $[0,\pi],$ so  is $q_1\mapsto \int_{\T}\frac{dp_1}{A^2 - b^2\cos^2 p_1}.$ Thus, $q_1\mapsto \hat \Delta(q,z)$ is strictly increasing. 

Now assume that $\hat z(q^0)<\hat\epsilon_{\min}(q^0)$ for some $q^0=(q_1 ^0, q_2^0, q_3^0)\in\T^3$ with $q_1^0\in [0,\pi].$ By the continuity of $\hat z,$ we have  $\hat z(q)<\hat\epsilon_{\min}(q)$ and $\hat\Delta(q,\hat z(q))=0$ for all $q$ in a small neighrborhood $N_{q_1^0}\times N_{q_2^0}\times N_{q_3}^0\subset\T^3$ of $q^0.$ Therefore, by the inverse function theorem in monotone case, $q_1\mapsto \hat z(q)$ is strictly increasing  in $N_{q_1^0}\cap [0,\pi].$

Since the map $q_1\mapsto \hat z(q)$ is locally increasing in $[0,\pi]$ near each $q\in \T^3$ with $q_1\in[0,\pi],$ the claim follows.

Since $\hat z$ is symmetric and even in each coordinates, this claim implies
\begin{equation}\label{axborot_tashkilot1}
\hat z(0) \le \hat z(q) \le \hat z(\vec \pi)\quad\text{for all $q\in\T^3$}.
\end{equation}
Moreover, by the minmax principle
$$
-\lambda=\hat\epsilon_{\min}(0)-\lambda \le \hat z(0) \quad\text{and}\quad  
\hat z(\vec\pi) \le \hat\epsilon_{\max}(\vec\pi)-\lambda = 3(1+\gamma+|1-\gamma|)-\lambda,
$$
and hence, combining these inequalities with \eqref{axborot_tashkilot1} we obtain \eqref{zgamlam_est_0pie}.

Finally, we prove \eqref{nice_estoa0s}.
For any  $z<0$ let us represent $\hat\Delta$ as  
$$
\hat \Delta(q,z) = 1- \frac{\lambda}{(2\pi)^3} \frac{1}{3(1+\gamma) - z} \int_{\T^3}\frac{dp}{ 1 - \xi_q(p)},
$$
where 
$$
\xi_q(p) = \sum_{i=1}^3 \frac{\sqrt{1+\gamma^2 + 2\gamma \cos q_i}}{3(1+\gamma) - z}\,\cos p_i.
$$
As $\|\xi_q\|_\infty<1$ and 
$$
\int_{\T^3} \xi_q(p)^{2m-1}dp = 0,\quad m\ge1,
$$
it follows that 
\begin{equation}
\hat \Delta(q,z) = 1 - \frac{\lambda}{3(1+\gamma) - z} \Big(1 + \frac{1}{(2\pi)^3}\sum_{m\ge1} \int_{\T^3}\xi_q(p)^{2m}dp \Big).
\end{equation}
From this inequality it follows that if $\lambda>3(1+\gamma),$ then
$\hat\Delta(q,3(1+\gamma)-\lambda) < 0.
$
Since we have also $\hat\Delta(q,z)\to1$ as $z\to-\infty,$ by the strict decreasing of $\hat \Delta(q,\cdot)$ in $(-\infty,\hat\epsilon(q)),$ the equation $\hat\Delta(q,z)=0$ has a unique solution $z=\hat z(q).$ Then 
$$
\hat\Delta(q,3(1+\gamma)-\lambda) < 0 =  \hat\Delta(q, \hat z(q)),
$$
and thus, by the strict decreasing of $z\mapsto \hat\Delta(q,z)$ in $(-\infty,0)$ we deduce 
\begin{equation*}
\hat z(q)<3(1+\gamma)-\lambda.
\end{equation*}
Next, as $|\xi_q(p)| \le \frac{3(1+\gamma)}{3(1+\gamma) - z}$ 
\begin{align*}
\hat\Delta(q,z)  \ge & 1 - \frac{\lambda}{3(1+\gamma) - z} \Big(1 + \sum_{m\ge1} \Big( \frac{3(1+\gamma)}{3(1+\gamma) - z}\Big)^{2m}\Big) \\
= & 1- \frac{\lambda}{3(1+\gamma) - z} \,\frac{(3(1+\gamma)-z)^2}{(3(1+\gamma)-z)^2 - 9(1+\gamma)^2}
\end{align*}
Now if $\lambda>3(1+\gamma),$ then 
$$
\hat\Delta\Big(q, 3(1+\gamma)- \lambda - \frac{9(1+\gamma)^2}{\lambda}\Big) =\frac{81(1+\gamma)^4\lambda^{-2}}{\lambda^2+9(1+\gamma)^2+81(1+\gamma)^4\lambda^{-2}}>0 = \hat\Delta(q, \hat z(q)).
$$
Thus, 
\begin{equation*}
 \hat z(q) > 3(1+\gamma)- \lambda - \frac{9(1+\gamma)^2}{\lambda}.
\end{equation*}
\end{proof}

\begin{remark}\label{rem:eigen_bor_all_lam}
Let $\gamma>0,$ $\lambda>0$ satisfy
\begin{equation*}
\lambda > (1+\gamma)\Big(\frac{1}{(2\pi)^3} \int_{\T^3} \frac{dp}{\epsilon (p)}
\Big)^{-1}.
\end{equation*}
Then for all $q\in\T^3,$ the function $\hat \Delta(q,\cdot)$ has a unique zero $\hat z_{\gamma,\lambda}(q)<\hat\epsilon_{\min}(q).$ Indeed, $\hat\Delta(q,z)=0$ admits a solution $z<\hat\epsilon_{\min}(q)$ if and only if 
\begin{align*}
\lambda > & \Big(\frac{1}{(2\pi)^3}\int_{\T^3} \frac{dp}{\sum_{i=1}^3 (1 + \gamma - \sqrt{1+\gamma^2+2\gamma\cos q_i} \cos p_i) - \hat\epsilon_{\min}(q)}\Big)^{-1} \\
= & \Big(\frac{1}{(2\pi)^3}\int_{\T^3} \frac{dp}{\sum_{i=1}^3 \sqrt{1+\gamma^2+2\gamma\cos q_i} (1-\cos p_i) }\Big)^{-1}.
\end{align*}
The last integral as a function of $q$ admits its minimum at $q=0,$ and thus, if 
$$
\lambda >\Big(\frac{1}{(2\pi)^3}\int_{\T^3} \frac{dp}{\sum_{i=1}^3 \sqrt{1+\gamma^2+2\gamma} (1-\cos p_i) }\Big)^{-1} = (1+\gamma) \Big(\frac{1}{(2\pi)^3}\int_{\T^3}\frac{dp}{\epsilon(p)}\Big)^{-1},
$$
then  for all $q$ the equation $\hat \Delta(q,\cdot)=0$ (and hence, $\Delta_{\gamma}(q,\cdot,\lambda)=0$) has a unique zero.
\end{remark}

\begin{remark}\label{rem:eigen_expression912}
Using an asymptotic analysis for the equation $\hat\Delta(q,\hat z_{\gamma,\lambda}(q))=0$ we can establish the asymptotics 
\begin{multline}\label{asymptotic}
\hat z_{\gamma,\lambda}(q) = -\lambda + 3(1+\gamma) - \frac{1}{2\lambda}\sum_{i=1}^3 (1+\gamma^2+2\gamma\cos q_i) \\
- \frac{1}{8\lambda^3} \Big(3\sum_{i=1}^3 (1+\gamma^2+2\gamma\cos q_i)^2 + 
2\Big(\sum_{i=1}^3(1+\gamma^2 + 2\gamma\cos q_i)\Big)^2\Big) 
+O\Big(\frac{1}{\lambda^5}\Big).
\end{multline}
\end{remark}

Note that in view of \eqref{ersiaosdx} 
\begin{equation}\label{w26tza1}
z_{\gamma,\lambda}(K,q) = \hat z_{\gamma,\lambda}(K-q)+\epsilon(q).
\end{equation}
Therefore, from Lemma \ref{lem:prop_z_gamlam} we get the following corollary.

\begin{corollary}\label{cor:estimos_two_pasrtos}
For $\lambda,\gamma>0$ and $K,q\in\T^3$ let $z_{\gamma,\lambda}(K,q)$ be given by  \eqref{lower_bound_hKk}. Then 
\begin{multline}
-\lambda \le z_{\gamma,\lambda}(K,K)-\epsilon(K) \le z_{\gamma,\lambda}(K,q)-\epsilon(q) \\
\le z_{\gamma,\lambda}(K,\vec\pi-K)-\epsilon(\vec\pi-K) \le -\lambda+3(1+\gamma+|1-\gamma|).
\end{multline}
In particular, when $K=0,$
\begin{equation}\label{baho_z_gl_0pi}
-\lambda \le z_{\gamma,\lambda}(0,0)\le z_{\gamma,\lambda}(0,q)-\epsilon(q) \\
\le z_{\gamma,\lambda}(0,\vec\pi)-\epsilon(\vec\pi) \le -\lambda+3(1+\gamma+|1-\gamma|). 
\end{equation}
Moreover, if $\lambda>3(1+\gamma),$ then 
\begin{equation}\label{asymp_zKp5e32}
\epsilon( q) - \lambda + 3(1+\gamma) - \frac{9(1+\gamma)^2}{\lambda} \le z_{\gamma,\lambda}(K,q) \le  \epsilon( q) - \lambda + 3(1+\gamma).
\end{equation}
\end{corollary}

Using the asymptotics \eqref{asymptotic} of $\hat z_{\gamma,\lambda}$ and the representation \eqref{w26tza1} of $z_{\gamma,\lambda}$ we can find an asymptotic formula for minimum and maximum points of $z_{\gamma,\lambda}.$

\begin{lemma}\label{lem:min_max_zgamlamK}
For any $\gamma>0$ there exists $\bar\lambda_\gamma>0$ such that for any $K\in\T^3$ and $\lambda>\bar\lambda_\gamma$ the function $q\mapsto z_{\gamma,\lambda}(q)$ has a unique maximum and unique minimum points $q_{\gamma,\lambda}^{\max}(K)$ and $q_{\gamma,\lambda}^{\min}(K).$ Moreover, both critical points are nondegenerate and 
$$
q_{\gamma,\lambda}^{\min}(K)=
\frac{\gamma}{\lambda}(\sin K_1,\sin K_2,\sin K_3) + O(\lambda^{-2}) 
$$
and 
$$
q_{\gamma,\lambda}^{\max}(K)=\vec\pi +
\frac{\gamma}{\lambda}(\sin K_1,\sin K_2,\sin K_3) + O(\lambda^{-2})  
$$
as $\lambda\to+\infty.$ Finally, if $K=0,$ then $q_{\gamma,\lambda}^{\max}(0) = \vec\pi$ and 
$q_{\gamma,\lambda}^{\min}(0) = 0.$
\end{lemma}

\begin{proof}
Let us assume that $\lambda>3(1+\gamma)$ so that by \eqref{nice_estoa0s} $\hat z_{\gamma,\lambda}(q)<0\le\hat  \epsilon_{\min}(q)$ is the unique zero of $\hat\Delta(q,\cdot).$ By \eqref{asymptotic} and \eqref{w26tza1} 
$$
z_{\gamma,\lambda}(K,q) = \epsilon(q) - \lambda + 3(1+\gamma) - \frac{1}{2\lambda}\sum_{i=1}^3 (1+\gamma^2+2\gamma\cos(K_i-q_i)) - \frac{g_{\gamma,\lambda}(K,q)}{\lambda^3},
$$
where $g_{\gamma,\lambda}(K,\cdot)$ is a real-analytic function of $\T^3,$ whose all derivatives are bounded, uniformly in $K$ and $\lambda.$ Thus, if $q_i$ is a critical point, then 
\begin{equation}\label{all_asd_qws}
\p_i z_{\gamma,\lambda}(K,q) = \sin q_i + \frac{\gamma}{\lambda}\sin (q_i-K_i) - \frac{\p_i g_{\gamma,\lambda}(K,q)}{\lambda^3} = 0\quad\text{for $i=1,2,3,$}
\end{equation}
where $\p_i=\frac{\p}{\p q_i}.$ Now let  $\theta_i\in (-\pi/2,\pi/2)$ be such that 
\begin{equation}\label{saylov_otdi_qoldi1}
\cos\theta_i = \frac{1 + \frac{\gamma}{\lambda}\cos K_i}{\sqrt{1+\frac{\gamma^2}{\lambda^2}+\frac{2\gamma}{\lambda}\cos K_i}},\quad \sin\theta_i = \frac{\frac{\gamma}{\lambda}\sin K_i}{\sqrt{1+\frac{\gamma^2}{\lambda^2}+\frac{2\gamma}{\lambda}\cos K_i}},\quad i=1,2,3.
\end{equation}
Thus, from \eqref{all_asd_qws}  we obtain 
\begin{equation}\label{sinqi_thetai12}
\sin(q_i-\theta_i) = 
\frac{\p_i g_{\gamma,\lambda}(K,q)}{\lambda^3\sqrt{1+\frac{\gamma^2}{\lambda^2}+\frac{2\gamma}{\lambda}\cos K_i}},\quad i=1,2,3.
\end{equation}
Let 
$$
C_\gamma:=\sup_{i=1,2,3,\,K\in\T^3,\,\lambda>3(1+\gamma)}\,\|\p_i g_{\gamma,\lambda}(K,\cdot)\|_\infty
$$
and assume $\lambda>2C_{\gamma}^{1/3}.$ Then from \eqref{sinqi_thetai12} and \eqref{saylov_otdi_qoldi1} we get either 
$$
q_i^a = \arcsin \frac{\gamma \sin K_i}{\lambda\sqrt{1+\frac{\gamma^2}{\lambda^2}+\frac{2\gamma}{\lambda}\cos K_i}} + \arcsin 
\frac{\p_i g_{\gamma,\lambda}(K,q^a)}{\lambda^3\sqrt{1+\frac{\gamma^2}{\lambda^2}+\frac{2\gamma}{\lambda}\cos K_i}}
$$
or 
$$
q_i^b = \pm\pi + \arcsin \frac{\gamma \sin K_i}{\lambda\sqrt{1+\frac{\gamma^2}{\lambda^2}+\frac{2\gamma}{\lambda}\cos K_i}} - 
\frac{\p_i g_{\gamma,\lambda}(K,q^b)}{\lambda^3\sqrt{1+\frac{\gamma^2}{\lambda^2}+\frac{2\gamma}{\lambda}\cos K_i}},
$$
where the signs are taken so that $q_i\in\T.$ Now using the Taylor series of $(1+u)^{-1/2}$ and $\arcsin u$ at $u=0$ we find the following asymptotics of $q_i^a$ and $q_i^b:$
\begin{equation}\label{def:q_ia432}
q_i^a= \frac{\gamma\sin K_i}{\lambda} - \frac{\gamma^2\sin K_i\cos K_i}{\lambda^2} + \frac{w_{\gamma,\lambda}^{i,1}(K,q)}{\lambda^3}
\end{equation}
or 
\begin{equation}\label{def:q_ib432}
q_i^b = \pm\pi + \frac{\gamma\sin K_i}{\lambda} - \frac{\gamma^2\sin K_i\cos K_i}{\lambda^2} + \frac{w_{\gamma,\lambda}^{i,2}(K,q)}{\lambda^3}
\end{equation}
for some real-analytic functions $w_{\gamma,\lambda}^i(K,\cdot)$ and $w_{\gamma,\lambda}^2(K,\cdot),$ bounded uniformly in $K$ and $\lambda.$ 

We claim that for large $\lambda,$ depending only on $\gamma,$ $q_i^a$ is the unique minimum and $q_i^b$ is the unique maximum of $z_{\gamma,\lambda}(K,\cdot).$ To show this we compute the Hessian of $z_{\gamma,\lambda}:$
$$
\p_{i}\p_j z_{\gamma,\lambda}(K,q) = 
\begin{cases}
-\frac{\p_i\p_j g_{\lambda,\gamma}(K,q)}{\lambda^3} & \text{if $i\ne j,$} \\
\cos q_i + \frac{\gamma}{\lambda}\cos (q_i-K_i) - \frac{\p_i\p_i g_{\gamma,\lambda}(K,q)}{\lambda^3} & \text{if $i=j.$}
\end{cases}
$$
As 
\begin{align*}
\cos q_i^a+\frac{\gamma}{\lambda}\cos (q_i^a-K_i) = & \Big(1+\frac{\gamma}{\lambda}\cos K_i\Big) \cos q_i^a + \frac{\gamma\sin K_i}{\lambda}\sin q_i^a \\
= & \sqrt{1+\frac{\gamma^2}{\lambda^2}+\frac{2\gamma}{\lambda}\cos K_i}\,\cos (q_i^a-\theta_i) \\
= & \sqrt{1+\frac{\gamma^2}{\lambda^2}+\frac{2\gamma}{\lambda}\cos K_i}\,\sqrt{1 - \frac{[\p_ig_{\gamma,\lambda}(K,q)]^2}{\lambda^6\Big(1+\frac{\gamma^2}{\lambda^2}+\frac{2\gamma}{\lambda}\cos K_i\Big)}},
\end{align*}
where in the last equalities we used \eqref{saylov_otdi_qoldi1} and \eqref{sinqi_thetai12},
we can write 
$$
\nabla^2 z_{\gamma,\lambda}(K,q^a) = 
I + \frac{\gamma}{\lambda}\,\diag(\cos K_1,\cos K_2,\cos K_3) + \frac{W_{\gamma,\lambda}(K)}{\lambda^2},
$$
where $I$ is the $3\times 3$-identity matrix, $\diag(x,y,z)$ is the diagonal matrix with diagonal entries $x,y,z\in\R,$ and $W_{\gamma,\lambda}(K)$ is bounded uniformly in $K$ and $\lambda.$ Set 
$$
\bar\lambda_\gamma:=\max\Big\{3(1+\gamma),\,  2C_\gamma^{1/3}, \,   \sup_{\lambda>2C_\gamma^{1/3},\,K\in\T^3}\, \sqrt{9\|W_{\gamma,\lambda}(K)\|}\Big\}.
$$
Then for any $\lambda>\bar\lambda_\gamma$
$$
\nabla^2 z_{\gamma,\lambda}(K,q^a) \ge I - \frac{\gamma}{\lambda} I - \frac{\|W_{\gamma,\lambda}(K)\|}{\lambda^2}I  \ge \frac{2+\gamma}{3+3\gamma}\,I
$$
Analogously, 
$$
\nabla^2 z_{\gamma,\lambda}(K,q^b) \le -\frac{2+\gamma}{3+3\gamma}\,I.
$$
Since $q^a$ and $q^b$ are only critical points of $z_{\gamma,\lambda}(K,\cdot),$ the estimates on the Hessians show that $q^a$ is the unique (nondegenerate) global minimum and $q^b$ is the unique (nondegenerate) global maximum. The asymptotics of $q_i^a$ and $q_i^b$ directly follows from identities  \eqref{def:q_ia432} and \eqref{def:q_ib432}.

Finally, assume that $K=0.$ Then by the evenness of $q\mapsto \hat z(q),$ we have 
$z_{\gamma,\lambda}(0,q) = \hat z(q) + \epsilon(q).$ Since $0$ resp. $\vec\pi$ is unique minimum resp. unique maximum for both $\hat z$ (see \eqref{axborot_tashkilot1}) and $\epsilon,$  we have $q_{\gamma,\lambda}^{\min}(0)= 0$ and 
$q_{\gamma,\lambda}^{\max}(0)= \vec\pi.$
\end{proof}

\subsection{Essential spectrum of $H_{\gamma,\lambda}(K)$}

\begin{theorem}[HVZ theorem]\label{teo:ess_spectrum_3}
For any $\lambda\ge0$ and $K\in\T^3,$
\begin{equation}\label{zu712s}
\sigma_\ess(H_{\gamma,\lambda}(K)) = \sigma(H_{\gamma,\lambda}^{(2)}(K)),
\end{equation}
where $H_{\gamma,\lambda}^{(2)}(K)$ is the channel operator, given by \eqref{channel_operator}.
\end{theorem}

Even though this kind of results are expected to exist in the literature, we could not find it and for the sake of completeness, we provide a full proof in the appendix. By  \eqref{shdzu67ec} and \eqref{zu712s}
\begin{equation}\label{esse_spectr_bestform}
\sigma_{\ess}(H_{\gamma,\lambda}(K))=[\tau_{\gamma,\min }(K,\lambda),
\tau_{\gamma,\max }(K,\lambda)]\cup [\cE_{\gamma,\min}(K), \cE_{\gamma,\max}(K)],
\end{equation}
where 
$$
\tau_{\gamma,\min }(K,\lambda) := \min_{q}\,z_{\gamma,\lambda}(K,q), \quad 
\tau_{\gamma,\max }(K,\lambda):=\max_{q}\,z_{\gamma,\lambda}(K,q).
$$
Observe that if 
\begin{equation*}
\lambda > \max_{q\in\T^3}\,\Big(\frac{1}{(2\pi)^3}\int_{\T^3} \frac{dp}{\cE_{\gamma,K}(p,q) - \cE_{\gamma,\min}(K)}\Big)^{-1},
\end{equation*}
then from the strict monotonicity of $\Delta_{\gamma,K}(q,\cdot,\lambda)$ we deduce 
$z_{\gamma,\lambda}(K,q)<\cE_{\gamma,\min}(K)$ for all $q\in\T^3$ so that $\tau_{\gamma,\max}(K,\lambda)<\cE_{\gamma,\min}(K).$ Hence, for such $\lambda$ there is a gap in the  essential spectrum of $H_{\gamma,\lambda}(K).$ 

When $K=0,$ in view of \eqref{zgamlam_est_0pie},
\begin{equation}\label{two_part_ess_K0}
[\tau_{\gamma,\min}(0,\lambda), \tau_{\gamma,\max}(0,\lambda)] = [z_{\gamma,\lambda}(0,0),z_{\gamma,\lambda}(0,\vec\pi)]\subset [-\lambda,-\lambda + 3(3+\gamma + |1-\gamma|)].
\end{equation}

\section{Birman-Schwinger-type principle}\label{sec:bsh_operator_definos}

In this section we introduce a Birman-Schwinger-type operator, a main tool in the proof of the existence and absence of the discrete spectra. 

Fix $K\in\T^3$ and $\lambda>0,$ and 
let us study the eigenvalue equation 
\begin{equation}\label{eigen_eq_H_K}
H_{\gamma,\lambda}(K) f = z f,\quad 0\ne f\in L^{2,a}((\T^3)^2),\quad z\in\R\setminus\sigma_\ess(H_{\gamma,\lambda}(K)).
\end{equation}
As $V$ are nonnegative, $z<\cE_{\gamma,\min}(K).$
Set 
\begin{equation*}
\phi(q):=V_{13}f(p,q) = (2\pi)^{-3}\int_{\T^3} f(s,q)ds
\end{equation*}
so that recalling the antisymmetricity of $f,$
$$
V_{23}f(p,q) = (2\pi)^{-3}\int_{\T^3} f(p,s)ds = -\phi(p).
$$
In particular, $\phi\in L^2(\T^3)$ and 
\begin{equation}\label{phi_has_0_mean}
\int_{\T^3}\phi(q)dq = (2\pi)^{-3}\int_{\T^3\times\T^3} f(s,q)dsdq = 0.
\end{equation}
Thus, \eqref{eigen_eq_H_K} is represented as 
\begin{equation}\label{equa_for_f}
f(p,q) = \frac{\lambda}{(2\pi)^{3}}\, \frac{\phi(q)-\phi(p)}{\cE_{\gamma,K}(p,q)-z}.
\end{equation}
Using \eqref{equa_for_f} in \eqref{eigen_eq_H_K} we get 
\begin{equation}\label{coming_to_det}
\phi(p)\Big(1-\frac{\lambda }{(2\pi)^3} \int_{\T^3}\frac{ds}
{\cE_{\gamma,K}(p,s)-z} \Big) + \frac{\lambda}{(2\pi)^3}
\int_{\T^3}\frac{\phi(s)ds} {\cE_{\gamma,K}(p,s)-z} = 0. 
\end{equation}
Note that 
$$
1-\frac{\lambda }{(2\pi)^3} \int_{\T^3}\frac{ds}
{\cE_{\gamma,K}(p,s)-z} = \Delta_{\gamma,K}(p,z,\lambda),
$$
where $\Delta_{\gamma,K}(\cdot,z,\lambda)$ is the Fredholm determinant, defined in \eqref{fredholm_determ}. As $z\notin\sigma_\ess(H_{\gamma,\lambda}(K)),$ it is neither in $[\cE_{\gamma,\min}(K),\cE_{\gamma,\max}(K)]$ nor an eigenvalue of $h_{\gamma,\lambda}(K,q)$ for any $q\in\T^3,$ Therefore, 
$$
\Delta_{\gamma,K}(q,z,\lambda) \quad \text{is}\quad 
\begin{cases}
>0 & \text{if $z<\tau_{\gamma,\min}(K,\lambda),$}\\
<0 & \text{if $z\in(\tau_{\gamma,\max}(K,\lambda),\cE_{\gamma,\min}(K))$}
\end{cases}
\quad\text{for all $q\in\T^3.$}
$$
Hence, setting 
\begin{equation*}
\psi:=\phi\sqrt{|\Delta_{\gamma,K}(\cdot,z,\lambda)|}
\end{equation*}
we can represent \eqref{coming_to_det} as a fixed point equation 
\begin{equation}\label{eigen_equas}
\psi(p) = -\frac{\lambda}{(2\pi)^3}\,\frac{\sign\Delta_{\gamma,K} }{\sqrt{|\Delta_{\gamma,K}(p,z,\lambda)|}} \int_{\T^3}\frac{\psi(s)ds} {(\cE_{\gamma,K}(p,s)-z) \sqrt{|\Delta_{\gamma,K}(s,z,\lambda)|}},
\end{equation}
where in the notation of \eqref{esse_spectr_bestform},
$\sign\Delta_{\gamma,K} = 1$ if $z< \tau_{\gamma,\min}(K,\lambda)$ and 
$-1$ if $z\in (\tau_{\gamma,\max}(K,\lambda),\cE_{\gamma,\min}(K)).$
Since $|\Delta_{\gamma,K}(\cdot,z,\lambda)|\ge \delta_{z,\lambda}>0,$ we have $\psi\in L^2(\T^3).$ Moreover, $\psi\ne0,$ otherwise, from its definition, $\phi=0,$ and hence, by \eqref{equa_for_f} $f=0,$ a contradiction. In view of \eqref{phi_has_0_mean}, $\psi$ satisfies the orthogonality assumption 
\begin{equation}\label{psi_ortho_det}
\int_{\T^3} \frac{\psi(s)ds}{\sqrt{|\Delta_{\gamma,K}(s,z,\lambda)|}} = 0.
\end{equation}

Conversely, assume that $0\ne \psi\in L^2(\T^3),$ satisfying the orthogonality condition \eqref{psi_ortho_det}, satisfies \eqref{eigen_equas}. Then the nonzero $L^2$-function $\phi:=\psi\,|\Delta_{\gamma,K}(\cdot,z,\lambda)|^{-1/2}$  solves \eqref{coming_to_det} and satisfies \eqref{phi_has_0_mean}. In view of the real-analyticity and nonconstancy of $\cE_{\gamma,K}$, $\phi$ is not constant. Then the function $f,$ given as in \eqref{equa_for_f}, is a nonzero $L^2$-function. By the symmetricity of $\cE_{\gamma,K}$ it follows that $f\in L^{2,a}((\T^3)^2).$ Then by \eqref{coming_to_det} $f$ satisfies \eqref{eigen_eq_H_K}.

These discussions yield the following theorem.

\begin{theorem}[A Birman-Schwinger-type principle]\label{teo:bsh_principle}
A number $z\in \R \setminus \sigma_\ess(H_{\gamma,\lambda}(K))$ is an eigenvalue of $H_{\gamma,\lambda}(K)$ if and only if $1$ is an eigenvalue of the operator 
$$
\cB_{\gamma,K}(z,\lambda) \phi(p):=-\frac{\lambda}{(2\pi)^3}\,\frac{\sign\Delta_{\gamma,K}}{\sqrt{|\Delta_{\gamma,K}(p,z,\lambda)|}} \int_{\T^3}\frac{\phi(s)ds} {(\cE_{\gamma,K}(p,s)-z) \sqrt{|\Delta_{\gamma,K}(s,z,\lambda)|}},\quad \phi\in L^2(\T^3),
$$
in the subspace 
\begin{equation}\label{def:H_delta}
\sH_\Delta:=\sH_{\gamma,\Delta}(K,z,\lambda):=\Big\{\psi\in L^2(\T^3):\,\, \int_{\T^3} \frac{\psi(s)ds}{\sqrt{|\Delta_{\gamma,K}(s,z,\lambda)|}} = 0\Big\}.
\end{equation}
Moreover, if $f\in L^{2,a}((\T^3)^2)$ is an eigenfunction of $H_{\gamma,\lambda}(K),$ associated to $z,$ then 
\begin{equation}\label{psi_defined_by_f}
\psi:= \sqrt{|\Delta_{\gamma,K}(\cdot,z,\lambda)|}\int_{\T^3}\,f(s,\cdot)ds
\end{equation}
is nonzero, belongs to $\sH_\Delta$ and an eigenfunction of $\cB_{\gamma,K}(z,\lambda),$ associated to $1.$ Conversely, if $0\ne\psi\in \sH_\Delta$ solves $ \cB_{\gamma,K}(z,\lambda) \psi = \psi,$ then the function 
\begin{equation}\label{f_defined_by_psi}
f(p,q):=\frac{\phi(q)-\phi(p)}{\cE_{\gamma,K}(p,q)-z},\quad \phi:=\frac{\psi}{\sqrt{|\Delta_{\gamma,K}(\cdot,z,\lambda)|}}
\end{equation}
is nonzero, belongs to $L^{2,a}((\T^3)^2)$ and satisfies $H_{\gamma,\lambda}(K) f= zf.$
Finally, the multiplicities of the eigenvalues $z$ and $1$ are the same.
\end{theorem}

\begin{proof}
We have already established the first assertion. To prove the last assertion, we choose orthonormal eigenfunctions $f_1,\ldots,f_n$ of $H_{\gamma,\lambda}(K),$ associated to $z,$ and define $\psi_i\in\sH_\Delta$ as in  \eqref{psi_defined_by_f} with $f=f_i$. By the first assertion of the theorem, each $\psi_i$ is a nonzero eigenfunction of $\cB_{\gamma,K}(z,\lambda),$ associated to $1.$ If $\psi_1,\ldots,\psi_n$ are linearly dependent, then so are the functions $\phi_i:=\psi_i|\Delta_{\gamma,K}(\cdot,z,\lambda)|^{-1/2},$ $i=1,\ldots,n.$ Then in view of the equality \eqref{equa_for_f} $f_1,\ldots,f_n$ are also linearly dependent, which contradicts to their orthogonality.
Conversely, if $\psi_1,\ldots,\psi_n\in \sH_\Delta$ are orthonormal eigenfunctions of $\cB_{\gamma,K}(z,\lambda),$ associated to $1,$ and $f_i$ are given by \eqref{f_defined_by_psi} with $\psi=\psi_i,$ then clearly, $f_i$ are linearly independent. 
Thus, the eigenspaces associated to $z$ in $L^{2,a}((\T^3)^2)$ and to $1$ in $\sH_\Delta$ have the same dimension.
\end{proof}

\section{Behaviour of $\cB_{\gamma,K}(z,\lambda)$ for large $\lambda$ and $|z|$}\label{sec:bsh_oper_for_large_lambda}

Unfortunately, the analysis of eigenvalues of $\cB_{\gamma,K}$ is too difficult in the general case. But we can say more when $\lambda$ and $|z|$ are large. In this case the operator can be represented as a sum of a finite rank integral operator and an operator of small norm.

Fix $\gamma>0$ and let $\bar\lambda_\gamma>0$ be given by Lemma \ref{lem:min_max_zgamlamK}. We also fix $K\in\T^3$ and $\lambda>\bar\lambda_\gamma.$ Let us study the behaviour of the projection of $\cB_{\gamma,K}(z,\lambda)$ onto $\sH_\Delta$ for large $\lambda$ and $|z|.$ Note that for any $z<\cE_{\gamma,\min}(K)$ we have 
$$
\frac{1}{\cE_{\gamma,K}(p,q) - z} = -\frac{1}{z} -
\frac{\cE_{\gamma,K}(p,q)}{z^2} + 
\Big[\frac{\cE_{\gamma,K}(p,q)}{z}\Big]^2\frac{1}{\cE_{\gamma,K}(p,q)-z}.
$$
These summands provide the principal and residual parts of $\cB_{\gamma,K}.$

Assuming $z\in (-\infty,\cE_{\gamma,\min}(K))\setminus [\tau_{\gamma,\min}(K,\lambda), \tau_{\gamma,\max}(K,\lambda)],$ let $\cB^r$ be the integral operator with the kernel 
$$
-\frac{\lambda}{(2\pi)^3}\frac{\sign\,\Delta_{\gamma,K}}{\sqrt{|\Delta_{\gamma,K}(p,z,\lambda)|}\sqrt{|\Delta_{\gamma,K}(p,z,\lambda)|}} \Big[\frac{\cE_{\gamma,K}(p,q)}{z}\Big]^2\frac{1}{\cE_{\gamma,K}(p,q)-z}.
$$
Note that for any $f\in L^2(\T^3)$
$$
|(\cB^r f,f)_{L^2}| \le \frac{\lambda}{(2\pi)^3} \frac{\cE_{\gamma,\max}(K)^2}{z^2\,(\cE_{\gamma,\min}(K)-z)} \Big(\int_{\T^3} \frac{f(p)dp}{\sqrt{|\Delta_{\gamma,K}(p,z,\lambda)|}}\Big)^2.
$$
Thus, 
\begin{equation*}
\|\cB^r\| \le  \frac{\lambda}{(2\pi)^3} \frac{\cE_{\gamma,\max}(K)^2}{z^2\,(\cE_{\gamma,\min}(K)-z)} \,\int_{\T^3} \frac{dp}{|\Delta_{\gamma,K}(p,z,\lambda)|}. 
\end{equation*}
Let us estimate $|\Delta_{\gamma,K}(p,z,\lambda)|$ from below:
\begin{align*}
|\Delta_{\gamma,K}(p,z,\lambda)| = & \Big|\frac{\lambda}{(2\pi)^3}\int_{\T^3}\frac{dq}{\cE_{\gamma,K}(p,q) - z_{\gamma,\lambda}(K,p)} -  \frac{\lambda}{(2\pi)^3}\int_{\T^3}\frac{dq}{\cE_{\gamma,K}(p,q) - z} \Big| \\
\ge & \frac{\lambda\,|z-z_{\gamma,\lambda}(K,p)|}{(\cE_{\gamma,\min}(K) - z_{\gamma,\lambda}(K,p))(\cE_{\gamma,\min}(K) - z) }.
\end{align*}
Moreover, recalling $\lambda>3(1+\gamma)$ and using \eqref{asymp_zKp5e32} we find 
$$
\cE_{\gamma,\min}(K) - z_{\gamma,\lambda}(K,p) \le 
\cE_{\gamma,\min}(K) - \epsilon(p) + \lambda - 3(1+\gamma)+ \frac{9(1+\gamma)^2}{\lambda}\le \cE_{\gamma,\min}(K) + \lambda
$$
provided $\lambda>3(2+\gamma).$ 
Thus, 
$$
\|\cB^r\| \le \frac{\cE_{\gamma,\max}(K)^2[\cE_{\gamma,\min}(K)  + \lambda]}{(2\pi)^3 z^2} \int_{\T^3} \frac{dp}{|z-z_{\gamma,\lambda}(K,p)|}.
$$
By Lemma \ref{lem:min_max_zgamlamK} $z_{\gamma,\lambda}(K,\cdot)$ admits a unique nondegenerate minimum $q_{\gamma,\lambda}^{\min}(K)$ and 
a unique nondegenerate maximum $q_{\gamma,\lambda}^{\max}(K),$ and  by the choice of $\lambda,$
$$
\tau_{\gamma,\min}(K,\lambda) = z_{\gamma,\lambda}(K,q_{\gamma,\lambda}^{\min}(K)),\quad 
\tau_{\gamma,\max}(K,\lambda) = z_{\gamma,\lambda}(K,q_{\gamma,\lambda}^{\max}(K)).
$$
Hence, 
$$
\int_{\T^3} \frac{dp}{|z-z_{\gamma,\lambda}(K,p)|}
\le \max\Big\{\int_{\T^3}\frac{dp}{z_{\gamma,\lambda}(K,q_{\gamma,\lambda}^{\max}(K)) - z_{\gamma,\lambda}(K,p)},\,\int_{\T^3}\frac{dp}{z_{\gamma,\lambda}(K,p) - z_{\gamma,\lambda}(K,q_{\gamma,\lambda}^{\min}(K)) }\Big\}.
$$
Clearly, last two integrals are bounded uniformly in $\lambda>\bar\lambda_\gamma$ and $K\in\T^3.$ Indeed, in view of the asymptotics of $z_{\gamma,\lambda}(K,p)$ in $\lambda,$ the only diverging term is a constant $-\lambda,$ the remaining $\lambda$-dependent terms are small. Moreover, by the dominated convergence theorem, both integrals are continuous functions of $K,$ and hence, using $0\le \cE_{\gamma,\min}(K)\le \cE_{\gamma,\max}(K) \le 3(2+\gamma)$ we deduce 
\begin{equation}\label{resi_2sd3}
\|\cB^r\| \le \frac{C_\gamma(1+\lambda)  }{z^2}.
\end{equation}
Now consider the principal part $\cB^p$ of $\cB_{\gamma,K},$ which is an integral operator with the kernel 
$$
-\frac{\lambda}{(2\pi)^3}\frac{\sign\,\Delta_{\gamma,K}}{\sqrt{|\Delta_{\gamma,K}(p,z,\lambda)|}\sqrt{|\Delta_{\gamma,K}(p,z,\lambda)|}} \Big[-\frac{1}{z} - \frac{\epsilon(p)+\epsilon(q)+\gamma\epsilon(K-p-q)}{z^2}\Big].
$$
Let $\cP_\Delta$ be the projection onto $\sH_\Delta,$ i.e.,
$$
\cP_\Delta f = f- (f,\phi_0)\phi_0,
$$
where $\phi_0:=\frac{c_0}{\sqrt{|\Delta_{\gamma,K} (\cdot,z,\lambda)|}}$ and 
$$
c_0:=\Big(\int_{\T^3}\frac{dq}{|\Delta_{\gamma,K}(q,z,\lambda)|}\Big)^{-1/2}.
$$
Using the explicit expression of the kernel of $\cB^p$ we observe that in computation of the projection 
$$
\tilde \cB:=\cP_\Delta\cB^p\cP_\Delta
$$ 
of $\cB^p$ onto $\sH_\Delta,$ the summands containing $1/z,$ $1/z^2,$ $\epsilon(p)/z^2$ and $\epsilon(q)/z^2$ vanish. Therefore, with a slight abuse of notation we assume that $\cB^p$ is the integral operator with the kernel 
$$
-\frac{\gamma\lambda}{(2\pi)^3}\frac{\sign\,\Delta_{\gamma,K}}{\sqrt{|\Delta_{\gamma,K}(p,z,\lambda)|}\sqrt{|\Delta_{\gamma,K}(q,z,\lambda)|}} \frac{\xi(K-p-q)}{z^2},
$$
where 
$$
\xi(p) = \sum_{i=1}^3 \cos p_i.
$$

\begin{lemma}\label{lem:eigens_BDelta}
For any $\lambda>\bar\lambda_\gamma$ and $z\in (-\infty,\cE_{\gamma,\min}(K))\setminus [\tau_{\gamma,\min}(K,\lambda),\tau_{\gamma,\max}(K,\lambda)]$ the operator $\tilde \cB$ has  three positive and three negative eigenvalues.
\end{lemma}

\begin{proof}
For shortness, let us write 
$$
\Delta(p):=|\Delta_{\gamma,K}(p,z,\lambda)|\quad\text{and}\quad \tilde c:=-\frac{\gamma\lambda}{(2\pi)^3}\,\frac{\sign\,\Delta_{\gamma,K}}{z^2}.
$$ 
Then by definition
$$
\cB^p f(p) = \tilde c \sum_{i=1}^3 \frac{\cos(K_i-p_i)}{\sqrt{\Delta(p)}}\int_{\T^3}\frac{\cos q_i f(q)dq}{\sqrt{\Delta(q)}} + \tilde c \sum_{i=1}^3 \frac{\sin(K_i-p_i)}{\sqrt{\Delta(p)}}\int_{\T^3}\frac{\sin q_i f(q)dq}{\sqrt{\Delta(q)}},
$$
and thus, 
\begin{multline*}
{\cB^p}^*\cB^p f(p) = \tilde c^2 \sum_{i=1}^3 \frac{\cos p_i}{\sqrt{\Delta(p)}} \int_{\T^3} \frac{\sum_{j=1}^3[a_{ij}\cos q_j+ b_{ij}\sin q_j]f(q)dq}{\sqrt{\Delta(q)}} \\ 
+  
\tilde c^2 \sum_{i=1}^3 \frac{\sin p_i}{\sqrt{\Delta(p)}} \int_{\T^3}\frac{\sum_{j=1}^3[b_{ij}\cos q_j+c_{ij}\sin q_j]f(q)dq}{\sqrt{\Delta(q)}},
\end{multline*}
where 
\begin{align*}
a_{ij} = \int_{\T^3} \frac{\cos(K_i-p_i)\cos(K_j-p_j)dp}{\Delta(p)},\\
b_{ij} = \int_{\T^3} \frac{\cos(K_i-p_i)\sin(K_j-p_j)dp}{\Delta(p)},\\
c_{ij} = \int_{\T^3} \frac{\sin(K_i-p_i)\sin(K_j-p_j)dp}{\Delta(p)}.
\end{align*}
Thus, ${\cB^p}^*\cB^p$ is a rank-six integral operator. Now  consider the projection $\cA:=\cP_\Delta {\cB^p}^*\cB^p\cP_\Delta.$ One can readily check that 
\begin{multline*}
\cA f(p) = \tilde c^2\sum_{i=1}^3 \psi_c(p_i) \int_{\T^3}\sum_{j=1}^3 [a_{ij}\psi_c(q_j) + b_{ij}\psi_s(q_j)]f(q)dq \\
+ \tilde c^2\sum_{i=1}^3 \psi_s(p_i) \int_{\T^3}\sum_{j=1}^3 [b_{ij}\psi_c(q_j) +  c_{ij}\psi_s(q_j)]f(q)dq,
\end{multline*}
where 
$$
\psi_c(p_i)= \frac{\cos p_i - C_i}{\sqrt{\Delta(p)}},\quad 
\psi_s(p_i)= \frac{\sin p_i - S_i}{\sqrt{\Delta(p)}},
$$
and 
\begin{equation}\label{lsuz78cdf}
C_i:= \int_{\T^3} \frac{\cos q_idq}{\Delta(q)}\Big(\int_{\T^3}\frac{dq}{\Delta(q)}\Big)^{-1}, 
\quad 
S_i:= \int_{\T^3} \frac{\sin q_idq}{\Delta(q)}\Big(\int_{\T^3}\frac{dq}{\Delta(q)}\Big)^{-1}.
\end{equation}
Let us show that $\cA$ has six positive eigenvalues. Since it has only rank 6 and $\psi_c(p_i)$ and $\psi_s(p_j)$ are linearly independent, the eigenvalue equation $\cA f(q) = \omega f(q)$ reduces to the study of the eigenvalue of some positive definite $6\times 6$-matrix $A,$ which surely have $6$ positive eigenvalues.

This observation yields that the operator $\cB^p$ has (exactly!) six nonzero eigenvalues for all $K\in\T^3,$ $\lambda>\bar\lambda_\gamma$ and $z\in (-\infty,\cE_{\gamma,\min}(K))\setminus [\tau_{\gamma,\min}(K,\lambda),\tau_{\gamma,\max}(K,\lambda)].$ Moreover, the eigenvalues continuously depend on $K,$ $\lambda$ and $z$ and never change the sign. Therefore, we can consider a particular value of the parameters. 

Let us assume that $K=0$ so that $S_i=0$ and $\psi_s$ are odd. Note also that $\psi_c$ are even.   First study the case $z\in (-\infty,\tau_{\gamma,\min}(0,\lambda)).$ In this case $\Delta_{\gamma,K}>0$ so that $\tilde c<0$ and 
\begin{equation}\label{bgdz67v}
\cP_\Delta\cB^p\cP_\Delta f(p) = \tilde c\sum_{i=1}^3 \psi_c(p_i)\int_{\T^3} \psi_c(q_i)f(q)dq -
\tilde c\sum_{i=1}^3 \psi_s(p_i) \int_{\T^3} \psi_s(q_i)f(q)dq.
\end{equation}
Thus, the subspaces $\sH_\Delta^e$ and $\sH_\Delta^o$ of even and odd functions in $\sH_\Delta, $ respectively, are invariant with respect to $\tilde\cB=\cP_\Delta\cB^p\cP_\Delta.$ One can readily check that 
$$
\tilde \cB\big|_{\sH_\Delta^e} f(p) = \tilde c\sum_{i=1}^3 \psi_c(p_i)\int_{\T^3} \psi_c(q_i)f(q)dq
$$
and 
$$
\tilde \cB\big|_{\sH_\Delta^o} f(p) = -\tilde c\sum_{i=1}^3 \psi_s(p_i)\int_{\T^3} \psi_s(q_i)f(q)dq.
$$
Since all summands are rank-one projections and $\tilde c<0,$ the rank-three operator $\tilde \cB\big|_{\sH_\Delta^e}$ is nonpositive definite and $\tilde \cB\big|_{\sH_\Delta^o}$ is nonnegative definite. Moreover, the former has exactly three negative eigenvalues and the latter has three positive eigenvalues. 
Thus, $\tilde \cB$ for $K=0$ (and hence, for all $K\in\T^3$) has exactly three positive and three negative eigenvalues.

Similarly, in case $z\in (\tau_{\gamma,\max}(0,\lambda),\cE_{\gamma,\min}(0))$ we have $\Delta_{\gamma,K}<0$ so that $\tilde c>0,$ and hence, recalling the representation \eqref{bgdz67v} of $\cP_\Delta \cB^p\cP_\Delta,$ we observe that $\tilde \cB\big|_{\sH_\Delta^e}\ge0$ and $\tilde \cB\big|_{\sH_\Delta^o}\le 0.$ Hence, again $\tilde \cB$ for $K=0$ (and hence, for all $K\in\T^3$) has exactly three positive and three negative eigenvalues.
\end{proof}

We are interested in the positive eigenvalues of $\tilde \cB,$ because when the residual part $\cP_\Delta\cB^r\cP_\Delta$  has small norm, those eigenvalues becomes the principal part of the positive eigenvalues of $\cB_{\gamma,K}(z,\lambda).$

Let us check whether the eigenvalues of $\cB^{r}:=\cB_{r,\gamma,K}(z,\lambda)$ does not vanish if we take large $\lambda$ and $z=z_\lambda.$ We provide the analysis for $\cA:=\cA(z,\lambda).$ To simplify the notations a bit, let 
us write $z(p):=z_{\gamma,\lambda}(K,p),$ $p_0:=q_{\gamma,\lambda}^{\min}(K)$ and $\cE(p,q):=\cE_{\gamma,K}(p,q).$

First consider the case $z<\tau_{\gamma,\min}(K,\lambda) = z(p_0).$ 
As we have seen earlier,  using $\Delta_{\gamma,K}(p,z(p),\lambda)=0$  we can write 
\begin{align*}
\Delta(p) = & \frac{\lambda}{(2\pi)^3} \int_{\T^3}\frac{(z(p) - z)ds}{(\cE(p,s) - z)(\cE(p,s) - z(p))}\\
= &  \frac{\lambda}{(2\pi)^3(-z)(-z(p))} \int_{\T^3}\frac{(z(p) - z)ds}{(1-\frac{\cE(p,s)}{z})(1-\frac{\cE(p,s) }{ z(p)})} \\
= & \frac{\lambda\,(z(p) - z)}{[-z]\,[-z(p_0)]\,[1- \frac{z(p)-z(p_0)}{-z(p_0)}]\, [1 + o(1)]}  , 
\end{align*}
where $o(1)\to0$ uniformly as $\lambda\to+\infty$ and in the last equality we used the asymptotics \eqref{asymp_zKp5e32} of $z(p):=z_{\gamma,\lambda}(K,p).$ To avoid dependence of $z$ on $\lambda, $  let us set $z= z(p_0) - \alpha$ for $\alpha>0.$ Then using the uniform boundedness of $z(p)-z(p_0)$ (especially, w.r.t. to $\lambda$) we get 
\begin{equation*}
\Delta(p) = \frac{\lambda(z(p)-z(p_0) + \alpha)}{[-z(p_0)+\alpha]\,[-z(p_0)]\,[1+o(1)]}.
\end{equation*}
Using this in the definitions of the constants $C_i$ and $S_i$ in \eqref{lsuz78cdf} 
we deduce 
\begin{align*}
C_i= \tilde C_i [1+o(1)],\quad \tilde C_i:=\int_{\T^3} \frac{\cos p_idp}{z(p)-z(p_0)+\alpha} \Big(\int_{\T^3}\frac{dp}{z(p)-z(p_0)+\alpha}\Big)^{-1},\\
S_i= \tilde S_i [1+o(1)],\quad \tilde S_i:=\int_{\T^3} \frac{\sin p_idp}{z(p)-z(p_0)+\alpha} \Big(\int_{\T^3}\frac{dp}{z(p)-z(p_0)+\alpha}\Big)^{-1}.
\end{align*}

remains bounded uniformly in $\lambda$ and $\alpha,$ and the constants $a_{ij},$ $b_{ij}$ and $c_{ij}$ have asymptotics 
\begin{align*}
a_{ij} = \frac{[-z(p_0)+\alpha]\,[-z(p_0)]\,[1+o(1)]}{\lambda} \tilde a_{ij},\quad \tilde a_{ij}:=\int_{\T^3}\frac{\cos (K_i-p_i)\cos (K_j-p_j)dp}{z(p)-z(p_0)+\alpha},\\
b_{ij} = \frac{[-z(p_0)+\alpha]\,[-z(p_0)]\,[1+o(1)]}{\lambda} \tilde b_{ij},\quad \tilde b_{ij}:=\int_{\T^3}\frac{\cos (K_i-p_i)\sin (K_j-p_j)dp}{z(p)-z(p_0)+\alpha},\\
c_{ij} = \frac{[-z(p_0)+\alpha]\,[-z(p_0)]\,[1+o(1)]}{\lambda} \tilde c_{ij},\quad \tilde c_{ij}:=\int_{\T^3}\frac{\sin (K_i-p_i)\sin (K_j-p_j)dp}{z(p)-z(p_0)+\alpha}
\end{align*}
as $\lambda\to+\infty.$ Finally, 
\begin{align*}
\psi_c(p_i) =  \sqrt{[-z(p_0)+\alpha]\,[-z(p_0)]\,[1+o(1)]}\tilde \psi_c(p_i),\quad 
\tilde \psi_c (p_i) = \frac{\cos p_i - \tilde C_i}{\sqrt{z(p)-z(p_0)+\alpha}},\\
\psi_s(p_i) =  \sqrt{[-z(p_0)+\alpha]\,[-z(p_0)]\,[1+o(1)]}\tilde \psi_s(p_i),\quad 
\tilde \psi_s (p_i) = \frac{\sin p_i - \tilde S_i}{\sqrt{z(p)-z(p_0)+\alpha}}.
\end{align*}
Inserting these asymptotics in the definition of $\cA$ we obtain 
$$
\cA = \tilde \cA + \cA_1,
$$
where $\|\cA_1\|\to0$ as $\lambda\to+\infty$ and  
\begin{multline*}
\tilde \cA f(p) = \frac{\gamma^2 \,[-z(p_0)]^2 }{(2\pi)^3[-z(p_0)+\alpha]^2}\,\Big(\sum_{i=1}^3 \tilde \psi_c(p_i) \int_{\T^3}\sum_{j=1}^3 [\tilde a_{ij}\tilde \psi_c(q_j) +\tilde  b_{ij} \tilde \psi_s(q_j)] f(q)dq \\
+ \sum_{i=1}^3\tilde  \psi_s(p_i) \int_{\T^3}\sum_{j=1}^3 [\tilde b_{ij}\tilde \psi_c(q_j) + \tilde c_{ij}\tilde \psi_s(q_j)]f(q)dq\Big).
\end{multline*}
Now we let $\lambda\to+\infty$ in $\tilde \cA.$ Then  by Lemma \ref{lem:min_max_zgamlamK} $p_0\to0$ so that by \eqref{asymptotic} and \eqref{w26tza1} $z(p)-z(p_0) \to \epsilon(p),$ and by \eqref{asymp_zKp5e32} $
\frac{\gamma^2 \,[-z(p_0)]^2 }{[-z(p_0)+\alpha]^2}\to\gamma^2.$ Therefore, $\cA = \gamma^2\cA_\alpha + o(1),$ where $\|o(1)\|\to0$ and 
\begin{multline*}
\cA_\alpha f(p) = \frac{1}{(2\pi)^3}\,\Big(\sum_{i=1}^3 \bar \psi_c(p_i) \int_{\T^3}\sum_{j=1}^3 [\bar a_{ij}\bar \psi_c(q_j) +\bar  b_{ij} \bar \psi_s(q_j)] f(q)dq \\
+ \sum_{i=1}^3\bar  \psi_s(p_i) \int_{\T^3}\sum_{j=1}^3 [\bar b_{ij}\bar \psi_c(q_j) + \bar c_{ij}\bar\psi_s(q_j)]f(q)dq\Big),
\end{multline*}
where the $\bar \psi_c,$ $\bar \psi_s,$ $\bar a_{ij},$ $\bar b_{ij}$ and $\bar c_{ij}$ are defined as their counterparts $\tilde \psi_c,$ $\tilde \psi_s,$ $\tilde a_{ij},$ $\tilde b_{ij}$ and $\tilde c_{ij},$ with $z(p)-z(p_0)$ replaced by $\epsilon(p).$
In particular, $\cA_\alpha$ is independent of $\gamma.$

Note that by the dominated convergence, we can even take $\alpha=0.$  Since $\cA_0$ has exactly six positive eigenvalues, the three positive eigenvalues of $\tilde \cB$ is uniformly bounded away from zero for all large $\lambda$ (depending only on $\gamma$).

Analogously, we will show that the three positive eigenvalues of $\tilde \cB$ for $z\in (\tau_{\gamma,\max}(K,\lambda),\cE_{\gamma,\min}(K))$ is uniformly bounded away from zero provided we choose $z$ of the form $\tau_{\gamma,\max}(K,\lambda) +\alpha$ for $\alpha>0.$ One can readily check that also in this case the principal part of $\cA$ is exactly $\gamma^2\cA_\alpha.$

\subsection{Discrete spectrum of $H_{\gamma,\lambda}(K)$}

In this section we prove Theorems \ref{teo:exist_eigen_H} and \ref{teo:exist_eigen_H_gap}. We start with the eigenvalues below the essential spectrum. 

\begin{proof}[Proof of Theorem \ref{teo:exist_eigen_H}]
Indeed, let $\tilde \beta_{\gamma,K}^i(z,\lambda),$ $i=1,2,3$ be the positive eigenvalues of $\tilde \cB.$ Then $[\tilde \beta_{\gamma,K}^i(z,\lambda)]^2$ are eigenvalues of $\tilde \cB^2 = \cA.$ Choosing $z=z(p_0)-\alpha$ and using the perturbation theory methods we find that the limit
$$
\tilde \beta_{\gamma,K}^i(\alpha):=\lim\limits_{\lambda\to+\infty} \tilde \beta_{\gamma,K}^i(z(p_0)-\alpha,\lambda)
$$
exists and $[\tilde \beta_{\gamma,K}^i(\alpha)]^2$ is a positive eigenvalue of $\gamma^2\cA_\alpha.$ Since $\cA_\alpha$ is independent of $\gamma,$ we have 
$$
\tilde \beta_{\gamma,K}^i(\alpha) = \gamma \tilde \beta_K^i (\alpha)
$$
for some $\tilde \beta_K^i (\alpha)>0.$
Moreover, as $\|\cA_\alpha\|\to0$ as $\alpha\to+\infty,$ we have $\tilde \beta_K^i (\alpha)\to0$ as $\alpha\to+\infty.$ Now we let 
$$
\gamma_1^{-1}:=\max_{i=1,2,3}\,\sup_{\alpha>0} \tilde \beta_K^i (\alpha)>0.
$$
If $\gamma\in (0,\gamma_1),$ then $\gamma \tilde \beta_{K}^i(\alpha)<1$ for all $\alpha\ge0$ and $i=1,2,3.$ In particular, taking $\lambda $ large, we conclude that all eigenvalues of
$\tilde \cB$ are strictly less than $1.$ Since $\cB_\Delta=\tilde \cB + \cP_\Delta \cB^r\cP_\Delta$ and the operator $\cP_\Delta \cB^r\cP_\Delta$ is norm-small for large $\lambda,$ it follows that $\cB_\Delta $ has no eigenvalues greater than $1.$

When $\gamma>\gamma_1,$ then there exists $\alpha>0$ such that $\tilde \beta_K^i(\alpha)>1$ for some $i=1,2,3.$ Then for sufficiently large $\lambda,$ the operator $\cB_\Delta$ has an eigenvalue greater than $1.$ Since $\cB_\Delta:=\cB_\Delta(z(p_0)-\alpha,\lambda)\to0$ as $\alpha\to+\infty,$ there exists $\alpha_\lambda$ such that $1$ is an eigenvalue of $\cB_\Delta(z(p_0)-\alpha_\lambda,\lambda).$ Now Theorem \ref{teo:bsh_principle} implies that $z(p_0)-\alpha_\lambda$ is an eigenvalue of $H_{\gamma,\lambda}(K).$

Similarly, setting 
$$
\tilde \gamma_1^{-1}:=\min_{1=1,2,3}\sup_{\alpha>0} \tilde \beta_K^i(\alpha)>0,
$$
we have that for any $\gamma>\tilde \gamma_1$ and for any large $\lambda$ the operator $H_{\gamma,\lambda}(K)$ has at least three eigenvalues below the essential spectrum.
\end{proof}


Repeating similar arguments we study the eigenvalues in the gap of the essential spectrum.

\begin{proof}[Proof of Theorem \ref{teo:exist_eigen_H_gap}]
By \eqref{asymp_zKp5e32}, \eqref{asymptotic} and Lemma \ref{lem:min_max_zgamlamK}
\begin{equation*}
\tau_{\gamma,\max}(K,\lambda) = z_{\gamma,\lambda}(K,q_{\gamma,\lambda}^{\max}(K)) = -\lambda + 3(3+\gamma) + O(\lambda^{-1})\quad\text{as $\lambda\to+\infty.$}
\end{equation*}
Thus, for $\lambda$ and $z\in [\tau_{\gamma,\max}(K,\lambda), \tau_{\gamma,\max}(K,\lambda)+ T_\lambda]$ by \eqref{resi_2sd3} 
we have 
$$
\|\cB^r(z,\lambda)\| \le \frac{C_\gamma(1+\lambda)}{z^2} \le \frac{C_\gamma(1+\lambda)}{(\lambda - 3(3+\gamma) - O(\lambda) - T_\lambda)^2} \to 0
$$
as $\lambda\to+\infty$ provided that 
$$
\lim\limits_{\lambda\to+\infty} \frac{\lambda}{T_\lambda} = +\infty.
$$
Thus, for large $\lambda$ the residual of $\cB_{\gamma,K}(z,\lambda)$ is small.

Now we study the principal part $\cB^p$ of $\cB_{\gamma,K}(z,\lambda).$ As in the proof of Theorem \ref{teo:exist_eigen_H},  let $\tilde \beta_{\gamma,K}^i(z,\lambda),$ $i=1,2,3$ be the positive eigenvalues of $\tilde \cB:=\cP_\Delta \cB^p\cP_\Delta.$ Then $[\tilde \beta_{\gamma,K}^i(z,\lambda)]^2$ are eigenvalues of $\tilde \cB^2 = \cA.$ Choosing $z=\tau_{\gamma,\max}(K,\lambda)+\alpha$ and using the perturbation theory methods we find that the limit
$$
\tilde \beta_{\gamma,K}^i(\alpha):=\lim\limits_{\lambda\to+\infty} \tilde \beta_{\gamma,K}^i(\lambda,\tau_{\gamma,\max}(K,\lambda)+\alpha)
$$
exists and $[\tilde \beta_{\gamma,K}^i(\alpha)]^2$ is a positive eigenvalue of $\gamma^2\cA_\alpha.$ Since $\cA_\alpha$ is independent of $\gamma,$ we have 
$$
\tilde \beta_{\gamma,K}^i(\alpha) = \gamma \tilde \beta_K^i (\alpha)
$$
for some $\tilde \beta_K^i (\alpha)>0.$ By the choice of the interval of $z,$ we have $\alpha \in [0,+\infty)$ and $\|\cA_\alpha\|\to0$ as $\alpha\to+\infty.$ In particular, $\tilde \beta_K^i (\alpha)\to0$ as $\alpha\to+\infty.$ 
Now the remaining assertions run as in the proof of Theorem \ref{teo:exist_eigen_H}.
\end{proof}


\section{On the discrete spectrum of $H_{\gamma,\lambda}(0)$}

In this section we assume that $K=0$ and study some spectral properties of $H_{\gamma,\lambda}(0)$ for large $\lambda.$ Unlike the general $K,$ we can describe the discrete spectrum of $H_{\gamma,\lambda}(0)$ more accurately in this case.

\subsection{The invariant subspaces of $\cB_{\gamma,0}(z,\lambda)$}

In this section we assume $\gamma>0,$ $\lambda>0,$ $K=0$ and provide a detailed analysis of the discrete spectrum of $\cB_{\gamma,0}(z,\lambda).$
We start by studying its invariant subspaces. Recall that in the proof of Lemma \ref{lem:eigens_BDelta} we have used invariant subspaces of $\tilde \cB$ -- the principal part of $\cB_{\gamma,0}.$

Note that  
$$
\tau_{\gamma,\min}(0,\lambda) \le \cE_{\gamma,\min}(0) = 0.
$$
As $\cE_{\gamma,0}(\cdot,\cdot)$ and $|\Delta_{\gamma,0}(\cdot,z,\lambda)|$ are even functions in $\T^3\times\T^3$ and $\T^3,$ respectively, the spaces $L^{2,e}(\T^3)$ and $L^{2,o}(\T^3)$ of (essentially) even and odd functions in $\T^3$ are invariant subspaces of 
$\cB_{\gamma,0}(z,\lambda).$ 
Set 
\begin{equation}\label{even_odd_parts01e}
\cB^e(z,\lambda) := \cB_{\gamma,0}(z,\lambda) \Big|_{L^{2,e}(\T^3)}
\quad\text{and}\quad 
\cB^o(z,\lambda) := \cB_{\gamma,0}(z,\lambda) \Big|_{L^{2,o}(\T^3)},
\end{equation}
where for simplicity we omit the dependence on $\gamma.$ 
One can readily check that $\cB^{e}$ and $\cB^o$ are integral operators with kernels 
$$
A^e(p,q,z,\lambda):=
\mp \frac{\lambda}{(2\pi)^3}\,\frac{E_c(p,q)-z}{\sqrt{|\Delta_\gamma(p,z,\lambda)|}(\cE_{\gamma,0}(p,q) - z)(\cE_{\gamma,0}(p,-q)-z)\sqrt{|\Delta_\gamma(q,z,\lambda)|}}
$$
and 
$$
A^o(p,q,z,\lambda):=
\pm \frac{\lambda}{(2\pi)^3}\,\frac{E_s(p,q)}{\sqrt{|\Delta_\gamma(p,z,\lambda)|}(\cE_{\gamma,0}(p,q) - z)(\cE_{\gamma,0}(p,-q)-z)\sqrt{|\Delta_\gamma(q,z,\lambda)|}},
$$
respectively, where we take the $+$ if $z<\tau_{\gamma,\min}(0,\lambda)$ and the $-$ if $z\in (\tau_{\gamma,\max}(0,\lambda),\cE_{\gamma,\min}(0)),$ 
$$
E_c(p,q):= \sum_{i=1}^3 (2+\gamma - \cos p_i - \cos q_i - \gamma \cos p_i\cos q_i)
$$
and 
\begin{equation}\label{E_sinus01}
E_s(p,q) = \gamma\sum_{i=1}^3\sin p_i\sin q_i
\end{equation}
are even and odd parts of $\cE_{\gamma,0}= E_c + E_s.$ 

First study the case of $z$ below the essential spectrum.

\begin{proposition}[$z$ is below the essential spectrum]\label{prop:property_even_odds}
Let $z<\tau_{\gamma,\min}(0,\lambda)\le0.$ Then
\begin{equation*}
\cB^e(z,\lambda)\le 0\quad \text{and}\quad \cB^o(z,\lambda)\ge0.
\end{equation*}
Moreover:
\begin{itemize}
\item[\rm(a)] writing  $L^{2,o}(\T^3)$ as the direct sum 
$$
L^{2,o}(\T^3) = \sH_1\oplus \sH_2 \oplus \sH_3 \oplus \sH_{123},
$$
where 
\begin{equation}\label{odd_subspaces0192}
\begin{gathered}
\sH_1:=L^{2,o}(\T)\otimes L^{2,e}(\T)\otimes L^{2,e}(\T),\\
\sH_2:=L^{2,e}(\T)\otimes L^{2,o}(\T)\otimes L^{2,e}(\T),\\
\sH_3:=L^{2,e}(\T)\otimes L^{2,e}(\T)\otimes L^{2,o}(\T),\\
\sH_{123}:=L^{2,o}(\T)\otimes L^{2,o}(\T)\otimes L^{2,o}(\T)
\end{gathered}
\end{equation}
are respectively the Hilbert spaces of all functions, essentially odd with respect to first, second, third and all coordinates, but even with respect to each remaining coordinates, we have that each $\sH_\zeta$ is an invariant subspace of $\cB^o(z,\lambda);$

\item[\rm(b)] the restrictions $\cB^o(z,\lambda)\big|_{\sH_\zeta}$ for $\zeta=1,2,3$ are mutually unitarily equivalent.

\end{itemize}

\end{proposition}

\begin{proof}
Note that $E_c\ge0.$  Summing the inequalities $\cE_{\gamma,0}(p,q)=E_c+E_s\ge0$ and $\cE_{\gamma,0}(p,-q)=E_c-E_s\ge0$  we get
$
|E_s| \le E_c.
$
Thus, for any $z<0$
\begin{equation*}
\frac{1}{\cE_{\gamma,0}-z} = \frac{1}{E_c-z} \frac{1}{1 + \frac{E_s}{E_c-z}}
= \frac{1}{E_c-z}\sum_{m\ge0} \Big(\frac{E_s}{E_c-z}\Big)^{2m} - \frac{1}{E_c-z}\sum_{m\ge0} \Big(\frac{E_s}{E_c-z}\Big)^{2m+1} =: E^e - E^o,
\end{equation*}
where both series converge uniformly in $\T^3\times\T^3.$ Clearly, 
$$
E^e(p,q) = \frac{E_c(p,q) - z}{(\cE_{\gamma,0}(p,q)-z)(\cE_{\gamma,0}(p,-q)-z)}
\quad \text{and} \quad 
E^o(p,q) = \frac{E_s(p,q)}{(\cE_{\gamma,0}(p,q)-z)(\cE_{\gamma,0}(p,-q)-z)},
$$
in particular, both $E^e$ and $E^o$ are symmetric, and 
\begin{equation}\label{evenlar0}
E^e(\pm p,\pm q)=E^e(\pm p,\mp q)=E^e(p,q)\quad \text{and} \quad E^o(\pm p,\pm q)=E^o(p,q) = -E^o(\pm p,\mp q).
\end{equation}
We represent $E_c$ as 
\begin{align*}
E_c(p,q) = \sum_{i=1}^3\Big[2+\gamma + \frac{1}{\gamma} - \Big(\frac{1}{\sqrt\gamma} + \sqrt\gamma\cos p_i\Big)\Big(\frac{1}{\sqrt\gamma} + \sqrt\gamma\cos q_i\Big)\Big] 
=:   \hat \gamma -
\hat E_c(p,q)\ge0,
\end{align*}
where
\begin{equation}\label{hat_E_cosinus01}
\hat \gamma :=
3\Big(2+\gamma + \frac{1}{\gamma}\Big),\quad 
\hat E_c(p,q) := \sum_{i=1}^3  \Big(\frac{1}{\sqrt\gamma} + \sqrt\gamma\cos p_i\Big)\Big(\frac{1}{\sqrt\gamma} + \sqrt\gamma\cos q_i\Big).
\end{equation}
Then using $|\hat E_c |\le \hat\gamma$ and $z<0$ we can write 
\begin{equation*}
\frac{1}{E_c -z} = \frac{1}{\hat \gamma -z} \cdot \frac{1}{1 - \frac{\hat E_c }{\hat \gamma-z}} = 
\frac{1}{\hat \gamma -z} \sum_{n\ge0 } \Big(\frac{\hat E_c }{\hat \gamma - z}\Big)^n.
\end{equation*}
As a result, 
\begin{equation*}
\frac{1}{E_c-z} \Big(\frac{E_s}{E_c-z}\Big)^{2m} = E_s^{2m} \Big( \frac{1}{\hat \gamma -z} \sum_{n\ge0 } \Big(\frac{\hat E_c }{\hat \gamma - z}\Big)^n \Big)^{2m+1}  
= \sum_{i_1,\ldots,i_{2m+1}\ge0} \frac{\hat E_c^{i_1+\ldots i_{2m+1}}E_s^{2m}}{(\hat \gamma-z)^{i_1+\ldots+i_{2m+1}+2m+1}}
\end{equation*}
and 
\begin{multline}
\label{D9zshed}
\frac{1}{E_c-z} \Big(\frac{E_s}{E_c-z}\Big)^{2m+1} = E_s^{2m+1} \Big( \frac{1}{\hat \gamma -z} \sum_{n\ge0 } \Big(\frac{\hat E_c }{\hat \gamma - z}\Big)^n \Big)^{2m+2} \\
= \sum_{i_1,\ldots,i_{2m+2}\ge0} \frac{\hat E_c^{i_1+\ldots i_{2m+2}}E_s^{2m+1}}{(\hat \gamma-z)^{i_1+\ldots+i_{2m+2}+2m+2}}
\end{multline}
which all series converge uniformly in $\T^3\times\T^3.$
Thus, both $\cB^e$ and $\cB^o$ are represented as a convergent sum
\begin{equation}\label{zaus76w}
\cB^e:= -
\sum_{m\ge0} \sum_{i_1,\ldots,i_{2m+1}\ge0} \cB_{2m,i_1,\ldots,i_{2m+1}}\quad\text{and}\quad 
\cB^o:=
 \sum_{m\ge0} \sum_{i_1,\ldots,i_{2m+2}\ge0} \cB_{2m+1,i_1,\ldots,i_{2m+2}}
\end{equation}
of integral operators $\cB_{m,i_1,\ldots,i_{m+1}}$ with symmetric kernels
\begin{equation}\label{kernel_sums921}
A_{m,i_1,\ldots,i_{m+1}}(p,q) = \frac{\lambda}{(2\pi)^3}\frac{1}{(\hat\gamma-z)^{i_1+\ldots,i_{m+1}+m+1}} \frac{\hat E(p,q)^{i_1+\ldots+i_{m+1}}E_s(p,q)^{m}}{\sqrt{|\Delta(p)|}\sqrt{\Delta(q)}},
\end{equation}
where for shortness we set $\Delta(p):=\Delta_{\gamma,0}(p,z,\lambda)$ and used $z<\tau_{\gamma,\min}(0,\lambda).$ In view of \eqref{E_sinus01} and \eqref{hat_E_cosinus01} the function $\hat E_c^lE_s^m$ is a finite sum of cosines and sines the form $\phi_j(p)\phi_j(q).$ 
Thus, in view of \eqref{kernel_sums921},
$$
(\cB_{m,i_1,\ldots,i_{m+1}}f,f)_{L^2} = \frac{\lambda}{(2\pi)^3}\frac{1}{(\hat\gamma-z)^{i_1+\ldots,i_{m+1}+m+1}} \sum_j \Big|\Big(f,\frac{\phi_j}{\sqrt{|\Delta|}}\Big)_{L^2}\Big|^2\ge0.
$$
Using this in \eqref{zaus76w} we deduce $\cB^e\le0$ and $\cB^o\ge0.$

The assertions (a) and (b) can be directly checked via a corresponding change of variables.
\end{proof}

Now consider $z$ belonging to the gap, assuming it exists (for instance $\lambda>3(3+\gamma+|1-\gamma|)$ so that by \eqref{two_part_ess_K0} $\tau_{\gamma,\max}(0,\lambda)<0=\cE_{\gamma,\min}(0)$).

\begin{proposition}[$z$ is in the gap]\label{prop:z_in_gap}
Assume that $z\in (\tau_{\gamma,\max}(0,\lambda), \cE_{\gamma,\min}(0)).$ Then 
\begin{equation*}
\cB^e(z,\lambda)\ge0 \quad\text{and}\quad  \cB^o(z,\lambda)\le0.
\end{equation*}
Moreover, writing $L^{2,e}(\T^3)$ as the direct sum 
$$
L^{2,e}(\T^3) = L^{2,e,s_2}(\T^3) \oplus L^{2,e,a_2}(\T^3) \oplus L^{2,e,\perp}(\T^3),
$$
where $L^{2,e,s_2}(\T^3)$ is the Hilbert space of all even functions on $\T^3$, symmetric with respect to permutations of any two variables $p_i$ and $p_j,$ $L^{2,e,a_2}(\T^3)$  is the Hilbert space of even functions on $\T^3,$ antisymmetric with respect to permutations of a fixed pair of coordinates, say $p_2$ and $p_3,$ and $L^{2,e,\perp}(\T^3)$ is the orthogonal complement of $L^{2,e,s_2}\oplus L^{2,e,a_2},$ we have that each $L^{2,e,\zeta}(\T^3),$ $\zeta\in\{s_2,a_2,\perp\}$ is an invariant subspace of $\cB^e(z,\lambda).$ 

\end{proposition}

\begin{proof}
As in the proof of Proposition \ref{prop:property_even_odds}, we represent $\cB^e$ and $\cB^o$ as a convergent sum
\begin{equation*}
\cB^e:= 
\sum_{m\ge0} \sum_{i_1,\ldots,i_{2m+1}\ge0} \cB_{2m,i_1,\ldots,i_{2m+1}}\quad\text{and}\quad 
\cB^o:=-
 \sum_{m\ge0} \sum_{i_1,\ldots,i_{2m+2}\ge0} \cB_{2m+1,i_1,\ldots,i_{2m+2}}
\end{equation*}
of nonnegative integral operators. This expansion implies $\cB^e\ge0$ and $\cB^o\le0.$ 

Consider the kernel $A^e(p,q)$ of $\cB^e,$ where for shortness, we are omitting dependence on $z$ and $\lambda.$ As $(p,q)\mapsto \cE_{\gamma,0}(p,q)$ is symmetric with respect to permutations of $(p_i,q_i)$ and even in each $(p_i,q_i),$
the function $p\mapsto \Delta_{\gamma,0}(p)$ is symmetric with respect to coordinate permutations and even in each coordinate. So, recalling also \eqref{evenlar0}, we conclude that $(p,q)\mapsto A^e(p,q)$ is also symmetric with respect to permutations of $(p_i,q_i)$ and even in each $(p_i,q_i)$. Therefore, $\cB^e$ maps each of the spaces $L^{2,e,s_2}(\T^3)$ and $L^{2,e,a_2}(\T^3)$ to itselfs. Finally, since $ \cB^e $ is self-adjoint, $\cB^e:L^{2,e,\perp}(\T^3) \to L^{2,e,\perp}(\T^3).$
\end{proof}

Note that $L^{2,o}(\T^3), L^{2,e,a_2}(\T^3),L^{2,e,\perp} (\T^3) \subset \sH_\Delta,$ where $\sH_\Delta:=\sH_{\gamma,\Delta}(0,z,\lambda)$ is given by \eqref{def:H_delta}. However, eigenfunctions of $\cB_{\gamma,0}(z,\lambda)$ in $L^{2,e,s_2}(\T^3)$ are not necessarily belong to $\sH_\Delta.$

\subsection{Discrete spectrum of $H_{\gamma,\lambda}(0)$ below the essential spectrum}

In this section we study discrete spectrum of $H_{\gamma,\lambda}(0)$ provided $\lambda$ is large and prove Theorem \ref{teo:eiegn_H_below}. In view of Theorem \ref{teo:bsh_principle}, $z$ is a discrete eigenvalue of $H_{\gamma,\lambda}(0)$ if and only if $1$ is an eigenvalue of $\cB_{\gamma,0}(z,\lambda).$ Therefore, using the decompositions of invariant subspaces, discussed earlier, we can study spectral properties of $\cB_{\gamma,0}(z,\lambda)$ for large $\lambda.$ 

We postpone the proof after some discussions. For  $\lambda>0$ and $z\le\tau_{\gamma,\min}(0,\lambda)<0$ consider the even and odd parts $\cB^e$ and $\cB^o$ of $\cB_{\gamma,0}$ given by \eqref{even_odd_parts01e}. By Proposition \ref{prop:property_even_odds} $\cB^e\le0$ and $\cB^o\ge0.$ Since we are interested in positive eigenvalues of $\cB_{\gamma,0},$  we study the behaviour of $\cB^o(z,\lambda)$ for large $\lambda$ (and hence,  by \eqref{two_part_ess_K0}, for small and very negative $z$). 

It turns out that for large $\lambda$ the operator $\cB^o$ is represented as a finite rank operator (preserving the invariant subspaces $\sH_\alpha$) plus a residual operator of small norm.
Indeed, in the notation of the proof of Proposition \ref{prop:property_even_odds}, using \eqref{D9zshed} we find 
$$
E^o = \frac{1}{E_c-z}\sum_{m\ge0} \Big(\frac{E_s}{E_c-z}\Big)^{2m+1} = \sum_{m\ge0} E_s^{2m+1} \Big(\frac{1}{\hat\gamma-z}\sum_{n\ge0}\Big(\frac{\hat E_c}{\hat\gamma-z}\Big)^{n}\Big)^{2m+2} = E^{o,p} + E^{o,r},
$$
where 
\begin{equation*}
E^{o,p}:=\frac{E_s}{(\hat\gamma - z)^2},
\end{equation*}
which is basically the summand of $E^o$ with $m=n=0,$ and $E^{o,r}=E^o-E^{o,p}$ are the principal and residual parts of $E^o.$ Let us find some $L^\infty$ estimate for the residual. Since $|\hat E_c| \le \hat\gamma,$ 
\begin{equation}\label{shdte637f}
\sum_{n\ge0} \Big(\frac{\hat E_c}{\hat\gamma-z}\Big)^{n}= \frac{\hat\gamma - z}{\hat\gamma-\hat E_c-z} =  \frac{\hat\gamma - z}{E_c-z}.
\end{equation}
Moreover, as $|E_s| \le E_c,$ 
\begin{align*}
\Big|\sum_{m\ge1} E_s^{2m+1} \Big(\frac{1}{\hat\gamma-z}\sum_{n\ge0}\Big(\frac{\hat E_c}{\hat\gamma-z}\Big)^{n}\Big)^{2m+2}\Big| = 
\frac{|E_s|^3}{(E_c-z)^2} \,\frac{1}{(E_c-z)^2-E_s^2}.
\end{align*}
Now using $E_c\pm E_s\ge0$ and $E_c\ge0,$ as well as $|E_s| \le 3\gamma,$ we obtain 
\begin{equation*}
\Big|\sum_{m\ge1} E_s^{2m+1}\Big(\frac{1}{\hat\gamma-z}\sum_{n\ge0}\Big(\frac{\hat E_c}{\hat\gamma-z}\Big)^{n}\Big)^{2m+2}\Big| \le \frac{27\gamma^3}{(-z)^4}. 
\end{equation*}
Similarly, using $E_c=\hat\gamma-\hat E_c\ge0,$ we have 
\begin{multline*}
\Big|\Big(\frac{1}{\hat\gamma-z}\sum_{n\ge0}\Big(\frac{\hat E_c}{\hat\gamma-z}\Big)^{n}\Big)^{2}- \frac{1}{(\hat\gamma-z)^2} \Big|
= \Big|\frac{1}{(E_c-z)^2}-\frac{1}{(\hat\gamma-z)^2}\Big| 
\\
=\frac{|2\hat E(\hat\gamma-z) + {\hat E_c}^2|}{(E_c-z)^2(\hat\gamma-z)^2}
\le \frac{2\hat\gamma}{(-z)^2(\hat\gamma-z)} + \frac{{\hat\gamma^2}}{(-z)^2(\hat\gamma-z)^2},
\end{multline*}
and therefore, 
\begin{equation}\label{good_est_Eores21}
|E^{o,r}(p,q,z )| \le  \frac{6\gamma \hat\gamma}{(-z)^2(\hat\gamma-z)} + \frac{3\gamma{\hat\gamma}^2}{(-z)^2(\hat\gamma-z)^2} +  \frac{27\gamma^3}{(-z)^4}.
\end{equation}
Now take $z<\tau_{\gamma,\min}(0,\lambda)$ and let $\cB^{o,p}(z,\lambda) $ and $\cB^{o,r}(z,\lambda)$ be the  integral operators with the kernels 
\begin{equation}\label{def:A_gamma_op}
A_\gamma^{o,p}(p,q,z,\lambda) :=\frac{\lambda}{(2\pi)^3}\,\frac{E^{o,p}(p,q,z )}{\sqrt{\Delta_\gamma(p,z,\lambda)} \sqrt{\Delta_\gamma(q,z,\lambda)}}
\end{equation}
and
$$
A_\gamma^{o,r}(p,q,z,\lambda) :=\frac{\lambda}{(2\pi)^3}\,\frac{E^{o,r}(p,q,z )}{\sqrt{\Delta_\gamma(p,z,\lambda)} \sqrt{\Delta_\gamma(q,z,\lambda)}},
$$  
respectively. As in Proposition \ref{prop:property_even_odds}, we can readily check that both operators are nonnegative. 

Let us show that the residual part $\cB^{o,r}$ of $\cB^{o}$ has small norm for large $\lambda.$ 

\begin{lemma}\label{lem:residual_small}
Let $\gamma>0$ and  
\begin{equation}\label{large_lambda0123}
\lambda>3(3+\gamma).
\end{equation}
Then there exists $C_1({\gamma})>3(3+\gamma)$ such that  %
\begin{align}\label{andalib_zoglar}
\|\cB^{o,r}(z,\lambda)\| \le  \frac{C_1(\gamma)}{\lambda}
\end{align}
for any  $z<\tau_{\gamma,\min}(0,\lambda).$
\end{lemma}

In view of \eqref{large_lambda0123} and  \eqref{nice_estoa0s}, $\sigma_{\ess}(H_{\gamma,\lambda}(0))$ consists of two disjoint segments $[\tau_{\gamma,\min}(0,\lambda),\tau_{\gamma,\max}(0,\lambda)] $ and $[\cE_{\gamma,\min}(0), \cE_{\gamma,\max}(0)].$ One possible choice of $C_1(\gamma)$ will be given in the proof.

\begin{proof}
Since $\cB^{o,r}$ is self-adjoint, by \eqref{good_est_Eores21}
\begin{align}\label{resid_estos12}
\|\cB^{o,r}(z,\lambda)\| = & \sup_{\|f\|_{L^2}=1} |(\cB^{o,r}(z,\lambda)f,f)_{L^2}| 
\le \frac{\lambda}{(2\pi)^3}\,\|E^{o,r}(\cdot,\cdot,z )\|_\infty  \sup_{\|f\|_{L^2}=1} \Big|\Big(|f|,\frac{1}{\sqrt{\Delta_\gamma(\cdot,z,\lambda)}}\Big)\Big|^2 \nonumber \\
= & \frac{\lambda}{(2\pi)^3}\,\|E^{o,r}(\cdot,\cdot,z )\|_\infty \int_{\T^3}\frac{dp}{\Delta_\gamma(p,z,\lambda)}.
\end{align}
Since we have already the $L^\infty$-estimate of $E^{o,r},$ we estimate $\Delta_\gamma(p,z,\lambda)$ from below. By assumption on $\lambda$ and Lemma \ref{lem:prop_z_gamlam}  $z_{\gamma,\lambda}(p):=\inf\sigma(h_{\gamma,\lambda}(0,p))=z_{\gamma,\lambda}(0,p)$ is the unique discrete eigenvalue of the effective one-particle operator $h_{\gamma,\lambda}(0,p).$ Thus, 
$\Delta_{\gamma}(p,z_{\gamma,\lambda}(p),\lambda)=0$ and hence,  
\begin{align} \label{gshezr64f}
\Delta_{\gamma}(p,z,\lambda) = & 
\Delta_{\gamma}(p,z,\lambda) - \Delta_{\gamma}(p,z_{\gamma,\lambda}(p),\lambda) \nonumber \\ 
= & \frac{\lambda}{(2\pi)^3}\int_{\T^3} \frac{(z_{\gamma,\lambda}(p) - z )\,dq}{(\cE_{\gamma,0}(p,q)-z_{\gamma,\lambda}(p)) (\cE_{\gamma,0}(p,q)-z )}.
\end{align}
In view of \eqref{baho_z_gl_0pi} 
$$
z_{\gamma,\lambda}(p) - z =
z_{\gamma,\lambda}(p) - z_{\gamma,\lambda}(0) + z_{\gamma,\lambda}(0) -  z \ge \epsilon(p)
$$
and 
\begin{equation}\label{estimos_Epq1a}
0 = \cE_{\gamma,\min}(0) \le \cE_{\gamma,0}(p,q) \le \cE_{\gamma,\max}(0) \le 2+2\gamma + \frac{1}{2\gamma} <\hat\gamma,\quad p,q\in\T^3.
\end{equation}
Moreover, by \eqref{zgamlam_est_0pie} $-z_{\gamma,\lambda}(p) \le -z_{\gamma,\lambda}(0)\le \lambda$ and thus, 
$$
\Delta_{\gamma}(p,z,\lambda) \ge \frac{\lambda}{(2\pi)^3}\int_{\T^3} \frac{\epsilon(p)\,dq}{(\hat\gamma + \lambda) (\hat\gamma-z )} = \frac{\lambda\epsilon(p)}{(\hat\gamma + \lambda) (\hat\gamma-z )}.
$$
Inserting this estimate of $\Delta_\gamma$ together with the $L^\infty$-estimate \eqref{good_est_Eores21} of $E^{o,r}$ in \eqref{resid_estos12}
we obtain
\begin{equation}\label{ahqd2sze6}
\|\cB^{o,r}(z,\lambda)\| \le \frac{\hat\gamma+\lambda}{(-z)^2}\Big(6\gamma\hat\gamma + \frac{3\gamma\hat\gamma}{\hat\gamma-z} + \frac{27\gamma^3(\hat\gamma-z)}{(-z)^2}\Big)\frac{1}{(2\pi)^3}\int_{\T^3}\frac{dq}{\epsilon(q)}.
\end{equation}
Finally, we observe that $z<z_{\gamma,\lambda}(0) < -\lambda +3(1+\gamma)$ (see \eqref{nice_estoa0s} for the latter). Thus, by assumption \eqref{large_lambda0123}, $-z>\lambda - 3(1+\gamma) \ge \frac{2\lambda}{3+\gamma}.$
Inserting this in \eqref{ahqd2sze6} we get \eqref{andalib_zoglar}, for instance, with 
$$
C_1(\gamma):=\Big(\frac{\hat\gamma(3+\gamma)}{12} + \frac{(3+\gamma)^2}{4}\Big)\Big(6\gamma\hat\gamma + \frac{3\gamma\hat\gamma}{\hat\gamma+6}+\frac{27\gamma^3\hat\gamma}{36}+\frac{27\gamma^3}{6}\Big)\frac{1}{(2\pi)^3}\int_{\T^3}\frac{dp}{\epsilon(p)}+3(3+\gamma).
$$
\end{proof}

Now consider the principal part $\cB^{o,p}$ of $\cB^o.$ By the explicit expression \eqref{def:A_gamma_op} of the kernel $A_\gamma^{o,p}$ of $\cB^{o,p},$ and the evenness and symmetry of $p_i\mapsto \Delta_\gamma(p,z,\lambda)$ it follows that each $\sH_\alpha,$ defined in \eqref{odd_subspaces0192}, is an invariant subspace of $\cB^{o,p}(z,\lambda).$ Moreover, $\cB^{o,p}(z,\lambda)\big|_{\sH_{123}}=0,$ and $\cB^{o,p}(z,\lambda)\big|_{\sH_{\alpha}}$ for $\alpha=1,2,3$ is a rank-one integral operator with the kernel
$$
A_\gamma^{o,p,\alpha}(p,q,z,\lambda) = \frac{\lambda}{(2\pi)^3(\hat\gamma-z)^2} \frac{\gamma\sin p_\alpha\sin q_\alpha}{\sqrt{\Delta_\gamma(p,z,\lambda)}\sqrt{\Delta_\gamma(q,z,\lambda)}}.
$$
It has a unique positive eigenvalue 
\begin{equation}\label{hetrz46f}
e_\gamma^{o,p,\alpha}(z,\lambda) = \frac{\lambda\gamma}{(2\pi)^3(\hat\gamma-z)^2} \int_{\T^3 } \frac{\sin^2p_\alpha\,dp}{\Delta_\gamma(p,z,\lambda)}.
\end{equation}
Since $\Delta_\gamma(\cdot,z,\lambda)$ is symmetric, $e_\gamma^{o,p,\alpha}$ is independent of $\alpha,$ i.e. $e_\gamma^{o,p}:=e_\gamma^{o,p,\alpha}$ is an eigenvalue of $\cB^{o,p}(z,\lambda)$ of multiplicity three. Clearly, $z\mapsto e_\gamma^{o,p}(z,\lambda)$ is strictly increasing in $(-\infty,z_{\gamma,\lambda}(0)],$ and 
\begin{equation}\label{cheksiz_e_baho123x}
\lim\limits_{z\to-\infty} e_\gamma^{o,p}(z,\lambda) = 0.
\end{equation}

\begin{lemma}[An estimate for $e_\gamma^{o,p}(z_{\gamma,\lambda}(0),\lambda)$]\label{lem:est_for_e_alpha192}
Assume that $\gamma,\lambda>0$ satisfy \eqref{large_lambda0123}. Then   there exists $C_2(\gamma)>0$ such that 
\begin{equation}\label{e_upper_baho12}
\frac{\gamma}{\gamma_1^0} - \frac{C_2(\gamma)}{\lambda} <e_{\gamma}^{o,p}(z_{\gamma,\lambda}(0),\lambda) <\frac{\gamma}{\gamma_1^0},
\end{equation}
where $\gamma_1^0:=\gamma_1(0)>0$ is given in \eqref{def:gamma_1}.
\end{lemma}

An explicit choice of $C_2(\gamma)$ is given in the proof.

\begin{proof}
By \eqref{gshezr64f} 
$$
\Delta_\gamma(p,z_{\gamma,\lambda}(0),\lambda) = \frac{\lambda}{(2\pi)^3} \int_{\T^3} \frac{(z_{\gamma,\lambda}(p) - z_{\gamma,\lambda}(0))dq}{(\cE_{\gamma,0}(p,q)-z_{\gamma,\lambda}(0))(\cE_{\gamma,0}(p,q)-z_{\gamma,\lambda}(p))}.
$$
Using \eqref{baho_z_gl_0pi}, \eqref{nice_estoa0s} and  \eqref{estimos_Epq1a} we can estimate $\Delta_\gamma(p,z_{\gamma,\lambda}(0),\lambda)$ as 
$$
\frac{\lambda\epsilon(p)}{(\hat\gamma - z_{\gamma,\lambda}(0))(\hat\gamma - z_{\gamma,\lambda}(p))} 
\le 
\Delta_\gamma(p,z_{\gamma,\lambda}(0),\lambda)
\le 
\frac{\lambda(\epsilon(p) + \frac{9(1+\gamma)^2}{\lambda})}{(- z_{\gamma,\lambda}(0))(- z_{\gamma,\lambda}(p))}.
$$
Thus,
\begin{multline*}
I_1:=\frac{\gamma\,(- z_{\gamma,\lambda}(0))}{(2\pi)^3(\hat\gamma - z_{\gamma,\lambda}(0))^2} 
\int_{\T^3}\frac{(- z_{\gamma,\lambda}(p))\sin^2p_1dp}{\epsilon(p) + \frac{9(1+\gamma)^2}{\lambda}} \\ 
< 
e_\gamma^{o,p}(z_{\gamma,\lambda}(0),\lambda) 
<
\frac{\gamma\,(\hat\gamma - z_{\gamma,\lambda}(0))}{(2\pi)^3(\hat\gamma - z_{\gamma,\lambda}(0))^2} 
\int_{\T^3}\frac{(\hat\gamma - z_{\gamma,\lambda}(p))\,\sin^2p_1dp}{\epsilon(p)}:=I_2.
\end{multline*}
By \eqref{zgamlam_est_0pie} 
$$
I_2 \le \frac{\gamma}{(2\pi)^3} 
\int_{\T^3}\frac{\sin^2p_1dp}{\epsilon(p)} = \frac{\gamma}{\gamma_1^0}.
$$
On the other hand, as $\lambda>3(3+\gamma),$ by \eqref{zgamlam_est_0pie} and  \eqref{nice_estoa0s} $ \lambda\ge -z_{\gamma,\lambda}(p) \ge \lambda-3(1+\gamma)-\epsilon(p) \ge \lambda-3(3+\gamma)$ for all $p\in\T^3.$ Therefore, 
$$
I_1\ge \frac{\gamma\,(\lambda - 3(3+\gamma))^2}{(2\pi)^3\,(\lambda + \hat\gamma)^2}
\int_{\T^3}\frac{\sin^2p_1dp}{\epsilon(p) + \frac{9(1+\gamma)^2}{\lambda}}=:I_1'.
$$ 
Note that 
\begin{align*}
\frac{\gamma}{\gamma_1^0} - I_1'  = & \Big(1-\frac{(\lambda - 3(3+\gamma))^2}{(\lambda + \hat\gamma)^2}\Big)\frac{\gamma}{(2\pi)^3} \int_{\T^3}\frac{\sin^2p_1dp}{\epsilon(p) + \frac{9(1+\gamma)^2}{\lambda}} \\
& + \frac{(\lambda - 3(3+\gamma))^2}{(\lambda + \hat\gamma)^2}\frac{\gamma}{(2\pi)^3} \Big(\int_{\T^3}\frac{\sin^2p_1dp}{\epsilon(p)}  - \int_{\T^3}\frac{\sin^2p_1dp}{\epsilon(p) + \frac{9(1+\gamma)^2}{\lambda}} \Big)\\
< & \frac{\gamma}{\gamma_1^0}\,\frac{6\lambda(3+\gamma)+9(3+\gamma)^2}{(\lambda+\hat\gamma)^2} + \frac{9\gamma(1+\gamma)^2}{\lambda}\,\frac{1}{(2\pi)^3} \int_{\T^3} \frac{\sin^2p_1dp}{\epsilon(p)^2}.
\end{align*}
Since the last integral is convergent, setting, for instance,
$$
C_2(\gamma):=\frac{9(3+\gamma)\gamma}{\gamma_1^0} + \frac{9\gamma(1+\gamma)^2}{(2\pi)^3}  \int_{\T^3} \frac{\sin^2p_1dp}{\epsilon(p)^2},
$$
we get \eqref{e_upper_baho12}.
\end{proof}

Now we are ready to establish the main result of the section.

\begin{proof}[Proof of Theorem \ref{teo:eiegn_H_below}]
By Proposition \ref{prop:property_even_odds}
$$
\sigma_\disc(\cB_{\gamma,0}(z,\lambda))\cap [0,+\infty) = \bigcup_{\alpha\in\{1,2,3,123\}}\,\sigma_{\disc}(\cB^o(z,\lambda)\big|_{\sH_\alpha}).
$$
As $L^{2,o}(\T^3)\subset \sH_\Delta,$ we do not need to project $\cB^o$ onto the $\sH_\Delta.$
Moreover, the principal part of $\cB^{o,p}(z,\lambda)\big|_{\sH_{123}}$ is zero, and hence, if $\lambda>C_1(\gamma),$ where $C_1(\gamma)>0$ is given by Lemma \ref{lem:residual_small}, 
then 
$$
\|\cB^{o}(z,\lambda)\big|_{\sH_{123}}\| = \|\cB^{o,r}(z,\lambda)\big|_{\sH_{123}}\|\le \frac{C_1(\gamma)}{\lambda}<1.
$$
Thus, $1$ cannot be an eigenvalue of $\cB^{o}(z,\lambda)\big|_{\sH_{123}}.$ 

Now consider the restriction $\cB^{o}(z,\lambda)\big|_{\sH_{1}},$ which is unitarily equivalent $\cB^{o}(z,\lambda)\big|_{\sH_{\zeta}}$ for $\zeta=2,3,$ so that its every eigenvalue $e_\gamma^o(z,\lambda)$ of multiplicity $k\ge1$ is an eigenvalue of $\cB^o(z,\lambda)$ of multiplicity $3k.$ Moreover, since $\cB^o = \cB^{o,p}+\cB^{o,r}$ and the norm of $\cB^{o,r}$ is small, by a simple perturbation arguments, 
$$
e_\gamma^o(z,\lambda) = e_\gamma^{o,p}(z,\lambda) + e_\gamma^{o,r}(z,\lambda), 
$$
where $e_\gamma^{o,p}(z,\lambda)$ is given by \eqref{hetrz46f} and $|e_\gamma^{o,r}(z,\lambda)|\le C_1(\gamma)/\lambda.$ Since $z\mapsto e_\gamma^{o,p}(z,\lambda)$ strictly increasing and continuous, in view of \eqref{cheksiz_e_baho123x}, the equation $e_\gamma^{o,p}(z,\lambda) = 1$ admits a solution $z<z_{\gamma,\lambda}(0)$ if and only if  
$e_\gamma^{o,p}(z_{\gamma,\lambda}(0),\lambda)>1.$

Now, in view of  \eqref{e_upper_baho12}, if $\gamma<\gamma_1^0,$ then for any $\lambda> \lambda_1(\gamma):=\frac{C_1(\gamma)}{1-\gamma/\gamma_1^0} $ the equation $e_\gamma^o(z,\lambda)=1$ does not have a solution $z<\tau_{\gamma,\min}(0,\lambda).$ Thus, by Theorem \ref{teo:bsh_principle} $H_{\gamma,\lambda}(0)$ has no discrete spectrum below the essential spectrum.

On the other hand, if $\gamma>\gamma_1^0,$ then for any $\lambda>\lambda_2(\gamma):= \frac{C_2(\gamma)}{\gamma/\gamma_1^0-1}$ the equation $e_\gamma^o(z,\lambda)=1$ admits a unique solution $z<\tau_{\gamma,\min}(0,\lambda).$ Since $1$ is a triple eigenvalue of $\cB^o(z,\lambda),$ by Theorem \ref{teo:bsh_principle} $z$ is a triple eigenvalue of $H_{\gamma,\lambda}(0)$ below the essential spectrum. Note that $H_{\gamma,\lambda}(0)$ cannot have other eigenvalues below the essential spectrum, otherwise, again by Theorem \ref{teo:bsh_principle} $\cB_{\gamma,0}(\cdot,\lambda)$ would have an eigenvalue $1,$ which contradicts to the uniqueness of $e_{\gamma}^o(\cdot,\lambda).$

Finally, by the explicit relationship \eqref{f_defined_by_psi} between the eigenvalues of $H_{\gamma,\lambda}(K)$ and $\cB_{\gamma,K}(z,\lambda),$ as well as the oddness of eigenvectors of $\cB^o$ it follows that any eigenfunction corresponding to an eigenvalue $z< \tau_{\gamma,\min}(0,\lambda)$ is in fact an odd function in $\T^3\times\T^3.$
\end{proof}

\subsection{Discrete spectrum of $H_{\gamma,\lambda}(0)$ in the gap}

In this section also we assume $\lambda$ satisfies \eqref{large_lambda0123} so that $\tau_{\gamma,\max}(0,\lambda) = z_{\gamma,\lambda}(\vec\pi) < 0=\cE_{\gamma,\min}(0),$ i.e., there is a gap in the essential spectrum of $H_{\gamma,\lambda}(0).$ By Theorem \ref{teo:bsh_principle} we still have that a $z\in (\tau_{\gamma,\max}(0,\lambda),0)$ is an eigenvalue of $H_{\gamma,\lambda}(0)$ of multiplicity $m\ge1$ if and only if $1$ is an eigenvalue of $\cB_{\gamma,0}(z,\lambda)$ of multiplicity $m,$ with an eigenvector in $\sH_\Delta,$ where $\sH_\Delta$ is defined in \eqref{def:H_delta} with $K=0.$ 
If we denote by $\cB^e$ and $\cB^o$ the even and odd parts of $\cB,$ then by Proposition \ref{prop:z_in_gap}  $\cB^e\ge0$ and $\cB^o\le0.$ Therefore, we study the positive eigenvalues of the operator $\cP_\Delta \cB^e\cP_\Delta,$ where $\cP_\Delta$ is the orthogonal projection onto $\sH_\Delta.$

Note that now a largeness of $\lambda$ does not necessarily lead to the largeness of $|z|.$ 
Thus, we need to distinguish two different regimes: $|z|$ is large and $|z|$ is small;
unfortunately, although we expect the eigenvalues of $\cB^e(z,\lambda)$ to lie close to the segments of the essential spectrum for large $\lambda,$  we are currently unable to say much about the case of intermediate or relatively small $|z|$ with respect to $\lambda.$

It is worth to mention that for large $|z|,$ as in the proof of Theorem \ref{teo:eiegn_H_below}, the behaviour of $\cB^e$ is completely determined by a finite rank operator.
Indeed, in the notation of the proof of Proposition \ref{prop:property_even_odds}, by \eqref{D9zshed}
$$
E^e = \frac{1}{E_c-z}\sum_{m\ge0} \Big(\frac{E_s}{E_c-z}\Big)^{2m} = \sum_{m\ge0} E_s^{2m} \Big(\frac{1}{\hat\gamma-z}\sum_{n\ge0}\Big(\frac{\hat E_c}{\hat\gamma-z}\Big)^{n}\Big)^{2m+1} = E^{e,p} + E^{e,r},
$$
where 
\begin{equation*}
E^{e,p}:=\frac{\hat E_c}{(\hat\gamma - z)^2},
\end{equation*}
which is basically the summand of $E^e$ with $m=0$ and $n=1,$ and $E^{e,r}=E^e-E^{e,p}$ are the principal and residual parts of $E^e.$ The summands of $E^{e,r}$ does not contain the summand $m=n=0,$ and thus, 
$$
\Big|\sum_{m\ge1} E_s^{2m} \Big(\frac{1}{\hat\gamma-z}\sum_{n\ge0}\Big(\frac{\hat E_c}{\hat\gamma-z}\Big)^n\Big)^{2m+1}\Big| = \frac{E_s^2}{(E_c-z)(E_c-E_s-z)(E_c+E_s-z)} \le \frac{9\gamma^2}{(-z)^3}
$$
and 
$$
\Big|\frac{1}{\hat\gamma-z} \sum_{n\ge2} \Big(\frac{\hat E_c}{\hat\gamma-z}\Big)^n\Big| = \frac{{\hat E_c}^2}{(\hat\gamma-z)^2 (E_c-z)} \le \frac{{\hat\gamma}^2}{(\hat\gamma-z)^2 (-z)}.
$$
Hence, 
\begin{equation*}
\|E^{e,r}(\cdot,\cdot,z)\|_\infty \le  \frac{9\gamma^2}{(-z)^3}+ \frac{{\hat\gamma}^2}{(\hat\gamma-z)^2 (-z)}
\end{equation*}

For any $z\in (z_{\gamma,\lambda}(\vec\pi),0)$ let $\cB^{e,p}(z,\lambda) $ and $\cB^{e,r}(z,\lambda)$ be the  integral operators with the kernels 
\begin{equation*}
A_\gamma^{e,p}(p,q,z,\lambda) :=\frac{\lambda}{(2\pi)^3}\,\frac{E^{e,p}(p,q,z )}{\sqrt{|\Delta_\gamma(p,z,\lambda|)} \sqrt{|\Delta_\gamma(q,z,\lambda)|}}
\end{equation*}
and
$$
A_\gamma^{e,r}(p,q,z,\lambda) :=\frac{\lambda}{(2\pi)^3}\,\frac{E^{e,r}(p,q,z )}{\sqrt{|\Delta_\gamma(p,z,\lambda)|} \sqrt{|\Delta_\gamma(q,z,\lambda)|}},
$$  
respectively. As in Proposition \ref{prop:property_even_odds}, we can readily check that both operators are nonnegative and $\cB_{\gamma,0}^e(z,\lambda) = \cB^{e,p}(z,\lambda) + \cB^{e,r}(z,\lambda).$ Moreover, since $L^{2,e}(\T^3)$ is not necessarily contained in $\sH_\Delta,$ we need to project both $\cB^{e,p}$ and $\cB^{e,r}$ onto $\sH_\Delta.$

We start with the following estimate on the residual part.

\begin{lemma}\label{lem:residual_B_e}
For any $\gamma,\lambda>0,$ satisfying \eqref{large_lambda0123} and $z\in (z_{\gamma,\lambda}(\vec\pi),0)$ one has  
\begin{equation}\label{small_part0}
\|\cP_\Delta \cB^{e,r}(z,\lambda)\cP_\Delta\| \le (\hat\gamma-z)(\hat\gamma + \lambda)\Big(\frac{9\gamma^2}{(-z)^3} + \frac{{\hat\gamma}^2}{(\hat\gamma-z)^2(-z)}\Big)\frac{1}{(2\pi)^3}\int_{\T^3}\frac{dp}{\epsilon(p)}.
\end{equation}
\end{lemma}

\begin{proof}
Since we are restricting ourselves to $\sH_\Delta,$  
\begin{align*}
\|\cP_\Delta \cB^{e,r}\cP_\Delta\| \le & \sup_{f\in\sH_\Delta,\,\|f\|=1} (\cB^{e,r}f,f)_{L^2} \le \frac{\lambda}{(2\pi)^3} \|E^{e,r}(\cdot,\cdot,z )\|_\infty \int_{\T^3}\frac{dp}{|\Delta_\gamma(p,z,\lambda)|}\\
\le & \frac{\lambda}{(2\pi)^3}\,\Big(\frac{9\gamma^2}{(-z)^3} + \frac{{\hat\gamma}^2}{(\hat\gamma-z)^2(-z)}\Big)\int_{\T^3}\frac{dp}{|\Delta_\gamma(p,z,\lambda)|}.
\end{align*}
By the choice of $z,$ 
$$
|\Delta_\gamma(p,z,\lambda)| = \frac{\lambda}{(2\pi)^3} \int_{\T^3}\frac{dq}{\cE_{\gamma,0}(p,q)-z}-1
$$
and hence increasing in $z.$ Moreover, 
\begin{align}\label{ashaudzee6}
|\Delta_\gamma(p,z,\lambda)|= & 
|\Delta_\gamma(p,z,\lambda) - \Delta_\gamma(p,z_{\gamma,\lambda}(p),\lambda)| \nonumber\\
= & \frac{\lambda}{(2\pi)^3} \int_{\T^3} \frac{(z - z_{\gamma,\lambda}(p))dq}{(\cE_{\gamma,0}(p,q)-z)(\cE_{\gamma,0}(p,q)-z_{\gamma,\lambda}(p))} 
\end{align} 
and in view of  \eqref{baho_z_gl_0pi} 
$$
z-z_{\gamma,\lambda}(p) > z_{\gamma,\lambda}(\vec\pi) - z_{\gamma,\lambda}(p) \ge \epsilon(\vec\pi) - \epsilon(p) = \epsilon(\vec\pi - p).
$$
Also, recalling $\cE_{\gamma,0}(p,q)\le \cE_{\gamma,\max}(0)\le \hat\gamma $ and $-z_{\gamma,\lambda}(p)\le -z_{\gamma,\lambda}(0)\le\lambda,$ 
$$
|\Delta_\gamma(p,z,\lambda)| \ge  \frac{\lambda \epsilon(\vec\pi-p)}{(\hat\gamma-z) (\hat\gamma + \lambda)}.
$$
Thus, 
$$
\|\cP_\Delta \cB^{e,r}(z,\lambda)\cP_\Delta\| \le 
(\hat\gamma-z) (\hat\gamma+\lambda)\,\Big(\frac{9\gamma^2}{(-z)^3} + \frac{{\hat\gamma}^2}{(\hat\gamma-z)^2(-z)}\Big)\frac{1}{(2\pi)^3}\int_{\T^3}\frac{dp}{\epsilon(p)}.
$$ 
\end{proof}

\begin{remark}
The estimate \eqref{small_part0} shows that to establish the smallness of $\cP_\Delta\cB^{e,r}\cP_\Delta$ we have to assume that 
$\lambda (-z)^{-2}\to0$ as $\lambda,|z|\to+\infty.$ In particular, $|z|$ should diverge faster than $\sqrt{\lambda}.$
\end{remark}

Now study the principal part $\cB^{e,p}$ of $\cB^e.$ 

\begin{lemma}\label{lem:eiegn_proj_ep}
The projection $\cB_\Delta^{e,p}:=\cP_\Delta\cB^{e,p}(z,\lambda)\cP_\Delta$  has three eigenvalues 
$$
e_{\gamma,1}^{e,p}(z,\lambda) = e_{\gamma,2}^{e,p}(z,\lambda) = \frac{\gamma\lambda}{2\cdot(2\pi)^3(\hat\gamma-z)^2} \int_{\T^3} \frac{(\cos p_1-\cos p_2)^2dp}{|\Delta_\gamma(p,z,\lambda)|}
$$
and 
$$
e_{\gamma,3}^{e,p}(z,\lambda) = \frac{\gamma\lambda}{3\cdot(2\pi)^3(\hat\gamma-z)^2} \int_{\T^3} \frac{(\cos p_1 + \cos p_2+\cos p_3 - 3\bar\beta)^2dp}{|\Delta_\gamma(p,z,\lambda)|},
$$
where 
$$
\bar\beta := \bar\beta(z,\lambda): = \int_{\T^3} \frac{\cos p_1\,dp}{|\Delta_\gamma(p,z,\lambda)|} \Big(\int_{\T^3}\frac{dp}{|\Delta_\gamma(p,z,\lambda)|}\Big)^{-1}.  
$$
\end{lemma}

\begin{proof}
Note that $\cB^{e,p}$ is of rank three. Since its kernel is explicit and $p\mapsto \Delta_\gamma(p,z,\lambda)$ is symmetric, we can readily compute the projection: 
$$
\cB_\Delta^{e,p}(z,\lambda) f(p) = \frac{\lambda}{(2\pi)^3} \int_{\T^3} \frac{\sum_{i=1}^3 (\frac{1}{\sqrt\gamma} + \sqrt\gamma\cos p_i - \beta)(\frac{1}{\sqrt\gamma} + \sqrt\gamma\cos q_i - \beta)f(q)dq}{(\hat\gamma - z)^2\,\sqrt{|\Delta_\gamma(p,z,\lambda)|} \sqrt{|\Delta_\gamma(q,z,\lambda)|}},
$$
where 
$$
\beta:=\beta_\gamma(z,\lambda):= \int_{\T^3} \frac{(\frac{1}{\sqrt\gamma} + \sqrt\gamma\cos p_1)dp}{|\Delta_\gamma(p,z,\lambda)|} \Big(\int_{\T^3}\frac{dp}{|\Delta_\gamma(p,z,\lambda)|}\Big)^{-1}.
$$
Let us introduce 
$$
\Phi_i(p):=\frac{\sqrt\lambda}{(2\pi)^{3/2}(\hat\gamma-z)} \frac{\frac{1}{\sqrt\gamma } + \sqrt\gamma\cos p_i - \beta}{\sqrt{|\Delta_\gamma(p,z,\lambda)|}},\quad i=1,2,3,
$$
so that 
$$
\cB_\Delta^{e,p}(z,\lambda) f(p) = \sum_{i=1}^3\Phi_i(p)\int_{\T^3}\Phi_i(q)f(q)dq,\quad f\in L^{2,e}(\T^3).
$$
Note that $\theta>0$ is an eigenvalue of $\cB_\Delta^{e,p}(z,\lambda)$ if and only if it is solves 
$$
\det 
\begin{pmatrix}
\theta - a_{11} & -a_{12} & -a_{13} \\
-a_{21} & \theta-a_{22} & -a_{23}  \\
-a_{31} & -a_{32} & \theta - a_{33}
\end{pmatrix}
=0,
$$
where
$$
a_{ij}:= \int_{\T^3} \Phi_i(p)\Phi_j(p)dp,
$$
Since $p\mapsto \Delta_\gamma(p,z,\lambda)$ is symmetric,
$a_{11}=a_{22}=a_{33}=a$ and $a_{ij}=b$ for $i\ne j.$ Thus, $\theta$ should solve 
$$
0= (\theta-a)^3 - 2b^3-3b^2(\theta-a) = (\theta -a + b)^2(\theta - a - 2b\Big).
$$
This equation has three solutions 
$$
\theta_1=\theta_2= a - b = \frac{1}{2}\int_{\T^3}(\Phi_1(p)-\Phi_2(p))^2dp  \quad\text{and}\quad \theta_3=a+2b
= \frac{1}{3}\int_{\T^3} (\Phi_1(p) +\Phi_2(p) + \Phi_3(p))^2dp.
$$
As $\beta = \frac{1}{\sqrt\gamma} + \sqrt\gamma \bar\beta,$ we have 
$$
\theta_1=\theta_2 = \frac{\gamma\lambda}{2\cdot(2\pi)^3(\hat\gamma-z)^2} \int_{\T^3} \frac{(\cos p_1-\cos p_2)^2dp}{|\Delta_\gamma(p,z,\lambda)|}
$$
and 
$$
\theta_3 = \frac{\gamma\lambda}{3\cdot(2\pi)^3(\hat\gamma-z)^2} \int_{\T^3} \frac{(\cos p_1 + \cos p_2+\cos p_3 - 3\bar\beta)^2dp}{|\Delta_\gamma(p,z,\lambda)|}.
$$
\end{proof}

Finally, we are ready to prove Theorem \ref{teo:eiegn_H_gap}.
Note that $z\mapsto e_{\gamma,i}^{e,p}(z,\lambda)$ is continuous in $(z_{\gamma,\lambda}(\vec\pi),0),$ but unlike its odd counterpart $e_{\gamma,i}^{o,p}(\cdot,\lambda),$ is not monotone. It is rather a product of one decreasing and one incresing functions, and hence, we expect some monotone decreasing near $z_{\gamma,\lambda}(\vec\pi),$ say in $(z_{\gamma,\lambda}(\pi),z_{\gamma,\lambda}(\pi)+ T_\lambda),$ where $T_\lambda$ is a nondecreasing function of $\lambda $ having sublinear growth with $T_\lambda\to+\infty$ as $\lambda\to+\infty$ and 
\begin{equation}\label{sgt56er}
\limsup\limits_{\lambda\to+\infty} \frac{\lambda}{T_\lambda}=+\infty.
\end{equation}
In particular, if $z$ ranges in $(z_{\gamma,\lambda}(\pi),z_{\gamma,\lambda}(\pi)+ T_\lambda),$ then by \eqref{small_part0} and \eqref{asymptotic} 
\begin{equation}
\label{agsr51x}
\|\cP_\Delta \cB^{e,r}(z,\lambda)\cP_\Delta\|\le \frac{C_3(\gamma)\lambda}{[z_{\gamma,\lambda}(\vec\pi) + T_\lambda]^2} \le \frac{C_3(\gamma)}{\lambda}
\end{equation}
for some $C_3(\gamma)>0$ and for all large $\lambda,$ and hence, for such $\lambda$ the residual operator has small norm.
As in the proof of Theorems \ref{teo:exist_eigen_H} and \ref{teo:exist_eigen_H_gap}  let us set 
\begin{equation}\label{new_definiz12}
z =z_{\gamma,\lambda}(\vec\pi) +\alpha \quad\text{for $\alpha\in (0,T_\lambda)$}
\end{equation}
and study $\alpha\mapsto \tilde e_{\gamma,i}^{e,p}(\alpha,\lambda):=e_{\gamma,i}^{e,p}(z_{\gamma,\lambda}(\vec\pi)+\alpha,\lambda).$ In this case \eqref{ashaudzee6} 
\begin{equation}\label{you_only_have12s}
|\Delta_\gamma(p,z_{\gamma,\lambda}(\vec\pi)+\alpha,\lambda)| = \frac{\lambda(z_{\gamma,\lambda}(\vec\pi) - z_{\gamma,\lambda}(p)+\alpha)}{[-z_{\gamma,\lambda}(\vec\pi)-\alpha]\,[-z_{\gamma,\lambda}(p)]} \Big[1+o(1)\Big],
\end{equation}
where $o(1)\to0$ as $\lambda\to0$ uniformly in $p.$
Then 
$$
\tilde e_{\gamma,1}^{e,p}(\alpha,\lambda) = 
\tilde e_{\gamma,2}^{e,p}(\alpha,\lambda) = 
\frac{\gamma\,[-z_{\gamma,\lambda}(\vec\pi)-\alpha]\,[-z_{\gamma,\lambda}(p)]}{2\cdot (2\pi)^3[\hat\gamma- z_{\gamma,\lambda}(\vec\pi)) -\alpha]^2} \int_{\T^3} \frac{(\cos p_1-\cos p_2)^2[1+o(1)]dp}{z_{\gamma,\lambda}(\vec\pi) -z_{\gamma,\lambda}(p) +\alpha}.
$$
By \eqref{asymptotic} and \eqref{w26tza1} the difference  $z_{\gamma,\lambda}(\vec\pi) -z_{\gamma,\lambda}(p) $ is uniformly bounded in $\lambda$ and converges to $\epsilon(\vec\pi)-\epsilon(p) = \epsilon(\vec\pi-p)$ as $\lambda\to+\infty.$ Using this we can readily show that 
\begin{equation}\label{e_12_even_psoasd}
\tilde e_{\gamma,i}^{e,p}(\alpha,\lambda) = \gamma e_{i}(\alpha)[1+o(1)],\quad i=1,2,
\end{equation}
where
$$
e_1(\alpha)=e_2(\alpha):=\frac{1}{2\cdot(2\pi)^3}\int_{\T^3} 
\frac{(\cos p_1-\cos p_2)^2dp}{e(\vec\pi-p) +\alpha}.
$$
Similarly, inserting \eqref{you_only_have12s} in the definition of $\bar\beta$ in Lemma \ref{lem:eiegn_proj_ep} we obtain 
$$
\bar\beta(z_{\gamma,\lambda}(\vec\pi)+\alpha,\lambda) = 
\bar\beta_\alpha\,[1+o(1)],\quad \bar\beta_\alpha:=\int_{\T^3}\frac{\cos p_1dp}{\epsilon(\vec\pi-p)+\alpha}\Big(\int_{\T^3}\frac{dp}{\epsilon(\vec\pi-p)+\alpha}\Big)^{-1},
$$
as $\lambda\to+\infty.$ Thus, for large $\lambda,$
\begin{equation}\label{e_3_even_psoasd}
\tilde e_{\gamma,3}^{e,p}(z_{\gamma,\lambda}(\vec\pi)+\alpha,\lambda) = e_3(\alpha) [1+o(1)],
\end{equation}
where
$$
e_3(\alpha) := 
\frac{1}{3\cdot(2\pi)^3} \int_{\T^3}\frac{(\cos p_1+\cos p_2+\cos p_3-3\bar\beta_\alpha)^2dp}{\epsilon(\vec\pi-p)+\alpha}.
$$
Combining \eqref{e_12_even_psoasd}, \eqref{e_3_even_psoasd}, \eqref{small_part0}, \eqref{new_definiz12} and \eqref{sgt56er} as well as the perturbation theory for linear operators, we conclude that for large $\lambda$ and for any $\alpha$ satisfying \eqref{new_definiz12} the operator $\cP_\Delta\cB_{\gamma,0}^e(z_{\gamma,\lambda}(\vec\pi)+\alpha,\lambda)\cP_\Delta$ has exactly three eigenvalues $e_{\gamma,i}^e(\alpha,\lambda)$ with principal parts $e_i(\alpha)[1+o(1)],$ i.e., 
\begin{equation}\label{eigen_Be_boars}
e_{\gamma,i}^e(\alpha,\lambda) = \gamma e_i(\alpha) [1+o(1)]\quad\text{as $\lambda\to+\infty,$}
\end{equation}
while by \eqref{agsr51x} all other eigenvalues (if any) have small norm. 

Notice that $e_i(\alpha)\to0$ as $\alpha\to+\infty.$ Since $e_1(\cdot)=e_2(\cdot)$ strictly decreases, if we set 
$$
\frac{1}
{\tilde \gamma_2}:=e_1(0):=\frac{1}{2\cdot(2\pi)^3}\int_{\T^3} 
\frac{(\cos p_1-\cos p_2)^2dp}{e(\vec\pi-p)} \approx 0.185237,\quad \text{i.e.}\quad \tilde \gamma_2 \approx 5.398489,
$$
then for any $\gamma<\tilde \gamma_2$ and for any $\alpha>0,$ we have $\gamma e_1(\alpha)=\gamma e_2(\alpha)<1.$ Thus, by \eqref{eigen_Be_boars} there exists $\lambda_7:=\lambda_7(\gamma)>0$ such that for any $\lambda>\lambda_7,$ $e_{\gamma,i}^e<1$ and thus, in this case two principal eigenvalues of $\cP_\Delta\cB^e\cP_\Delta$ are less than $1.$ 
On the other hand, for any $\gamma>\tilde \gamma_2,$ we can find $\lambda_8:=\lambda_8(\gamma)>0$ such that for any $\lambda>\lambda_8$ we can choose $\alpha:=\alpha_{\gamma,\lambda}$ such that 
$e_{\gamma,i}^e(\alpha,\lambda)=1.$ In this case, $1$ is an eigenvalue of $\cP_\Delta\cB^e\cP_\Delta$ at least of multiplicity two.

Similarly, set 
$$
\frac{1}{\gamma_2^0}:=\sup_{\alpha>0}\,e_3(\alpha) \ge e_3(0)\approx 0.340538 > e_1(0)=\frac{1}{\tilde \gamma_2^0}.
$$
Then for $\gamma<\gamma_2^0$ we can find $\lambda_9:=\lambda_9(\gamma)>0$ such that for $\lambda>\lambda_9$ the third principal eigenvalue $e_{\gamma,3}^e$ of $\cP_\Delta\cB^e\cP_\Delta$ is less than $1.$ Similarly, for $\gamma>\gamma_2^0$ we can find $\lambda_{10}:=\lambda_{10}(\gamma)>0$ such that for $\lambda>\lambda_{10}$ there exists at least one $\alpha:=\alpha_{\gamma,\lambda}>0$ such that $e_{\gamma,3}^{e}(\alpha,\lambda)=1;$ indeed, due to the lack of the monotonicity, this equation may have more solutions. 
Combining these observations with the Birman-Schwinger-type principle (Theorem \ref{teo:bsh_principle}) we obtain the following result. 

\begin{proposition}\label{prop:gap_eigenchuks}
Let $\tilde\gamma_2^0>\gamma_2^0>0$ be as above. Then: 
\begin{itemize}
\item[\rm(a)] for any $\gamma<\gamma_2^0$ there exists $\lambda_{11}:=\lambda_{11}(\gamma)>0$ such that for any $\lambda>\lambda_{11}$ the operator $H_{\gamma,\lambda}(0)$ has no eigenvalues in the portion $[z_{\gamma,\lambda}(\vec\pi),z_{\gamma,\lambda}(\vec\pi)+T_\lambda]$ of the gap;

\item[\rm(b)] for any $\gamma\in(\gamma_2^0,\tilde \gamma_2^0)$ there exists $\lambda_{12}:=\lambda_{12}(\gamma)>0$ such that for any $\lambda>\lambda_{12}$ the operator $H_{\gamma,\lambda}(0)$ has at least one eigenvalue in $[z_{\gamma,\lambda}(\vec\pi),z_{\gamma,\lambda}(\vec\pi)+T_\lambda];$

\item[\rm(c)] for any $\gamma>\tilde \gamma_2^0$ there exists $\lambda_{13}:=\lambda_{13}(\gamma)>0$ such that for any $\lambda>\lambda_{13}$ the operator $H_{\gamma,\lambda}(0)$ has at least three eigenvalues in $[z_{\gamma,\lambda}(\vec\pi),z_{\gamma,\lambda}(\vec\pi)+T_\lambda].$
\end{itemize}
\end{proposition}

Finally, by the explicit relationship \eqref{f_defined_by_psi} between the eigenvalues of $H_{\gamma,\lambda}(K)$ and $\cB_{\gamma,K}(z,\lambda),$ as well as the evenness of eigenvectors of $\cB^e$ it follows that any eigenfunction corresponding to an eigenvalue $z\in (z_{\gamma,\lambda}(\vec\pi),z_{\gamma,\lambda}(\vec\pi)+T_\lambda)$ is in fact an even function in $\T^3\times\T^3.$

\begin{remark}
$\,$
\begin{itemize}
\item A direct computation shows that for any $\alpha\ge\gamma,$
$$
\gamma e_1(\alpha) \le  \gamma e_1(\gamma) <1.
$$
Thus, for any $\gamma>\tilde\gamma_2^0$ and $\lambda>\lambda_{13}$ two eigenvalues of $H_{\gamma,\lambda}(0)$ in the gap $[z_{\gamma,\lambda}(\vec\pi),z_{\gamma,\lambda}(\vec\pi)+T_\lambda]$ belong in fact $[z_{\gamma,\lambda}(\vec\pi),z_{\gamma,\lambda}(\vec\pi)+\gamma].$ 

\item Similarly, since $|\bar\beta_\alpha|<1$ in Lemma \ref{lem:eiegn_proj_ep}, we can compute 
$$
e_3(\alpha) < \frac{7}{2\alpha}<\frac{1}{\gamma}
$$
provided that $\alpha>3.5\gamma.$
Thus, for $\gamma\in(\gamma_2^0,\tilde \gamma_2^0)$ and $\lambda>\lambda_{12}$ one eigenvalue of $H_{\gamma,\lambda}(0)$ lying in $[z_{\gamma,\lambda}(\vec\pi),z_{\gamma,\lambda}(\vec\pi)+T_\lambda]$ in fact belongs to the interval $[z_{\gamma,\lambda}(\vec\pi),z_{\gamma,\lambda}(\vec\pi)+3.5\gamma].$

\end{itemize}

\end{remark}

\appendix 
\section{Proof of Theorem \ref{teo:ess_spectrum_3}}

Let $H_{\gamma,\lambda}^{(3)}(K)$ be the extension of $H_{\gamma,\lambda}(K)$ to whole $L^2((\T^3)^2)$ (essentially given by the same formulas) so that by the Hunziker-van Winter-Zhislin theorem
\begin{equation*}
\sigma_\ess(H_{\gamma,\lambda}^{(3)}(K)) = \sigma(H_{\gamma,\lambda}^{(2)}(K)).
\end{equation*}
Since $H_{\gamma,\lambda}(K)$ is the restriction of $H_{\gamma,\lambda}^{(3)}(K)$ onto $L^{2,a}((\T^3)^2),$ we have 
$$
\sigma_\ess(H_{\gamma,\lambda}(K)) \subset \sigma_\ess(H_{\gamma,\lambda}^{(3)}(K)) = \sigma(H_{\gamma,\lambda}^{(2)}(K))
$$
To prove the converse inclusion we construct a test sequence as follows. For $p\in\T^3$ and $r>0,$ let $B_r(p):=\{q\in\T^3:\,|q-p|<r\}.$

Assume first that $z_0\in [\cE_{\gamma,\min}(K),\cE_{\gamma,\max}(K)]$ so that there exists $p_0,q_0\in\T^3$ such that $\cE_{\gamma,K}(p_0,q_0)=z_0.$  Define the open sets
$$
U_k:=B_{2^{-k}}(p_0)\setminus  \cl{B_{2^{-k-1}}(p_0)} \quad\text{and}\quad W_k:=B_{2^{-k-1}}(q_0)\setminus \cl{ B_{2^{-k-2}}(q_0) },\quad k\ge1.
$$
Note that $U_k\cap W_k=\emptyset$ for all large $k.$ Indeed, if $p_0\ne q_0,$ then for large $k$ (depending on $|p_0-q_0|$) the balls $B_{2^{-k}}(p_0)$ and $B_{2^{-k-1}}(q_0)$ becomes disjoint. Otherwise, $U_k$ and $W_k$ are two disjoint annuli.
Set 
$$
\psi_k(p,q):= \frac{\chi_{U_k}^{}(p)\chi_{W_k}^{}(q)}{\sqrt{|U_k|}\sqrt{|W_k|}},\quad 
\phi_k(p,q):= \frac{\psi_k(p,q) - \psi_k(q,p)}{\sqrt2}.
$$
Clearly, $\phi_k\in L^{2,a}((\T^3)^2),$ and by the choice of $U_k$ and $W_k,$ $\phi_k$ and $\phi_l$ are orthonormal for all large $k,l.$ Then for such $k$ 
$$
(V_{\alpha3}\phi_k,\phi_k) = \frac{|U_k| + |W_k|}{\sqrt2}\to0,
$$
and thus, by the nonnegativity of $V_{\alpha3}$ we get $\|V_{\alpha3}\phi_k\|_{L^2}\to0.$ Moreover,
\begin{align*}
\|(H_{\gamma,0}(K) - z_0I)\phi_k\|_{L^2}^2 = & \|(H_{\gamma,0}(K) - z_0I)\psi_k\|_{L^2}^2  \\
\le & \frac{1}{|U_k|\,|W_k|} \int_{U_k\times W_k}|\cE_{\gamma,0}(p,q)-\cE_{\gamma,K}(p_0,q_0)|^2dpdq \to0
\end{align*}
as $\cE_{\gamma,K}$ is continuous. Hence,  $\|(H_{\gamma,\lambda}(K) - z_0)\phi_k\|_{L^2}^2\to0,$ and by the Weyl's criterion $z_0\in \sigma_\ess(H_{\gamma,\lambda}(K)).$

Next, assume that $z_0\in\sigma(H_{\gamma,\lambda}^{(2)}(K))\setminus [\cE_{\gamma,\min}(K),\cE_{\gamma,\max}(K)].$ In view of \eqref{shdzu67ec} there exists $q_0\in\T^3$ such that $z_0<\cE_{\gamma,\min}(K)\le \cE_{\gamma,\min}(K,q_0)$ is a discrete eigenvalue  of the effective one-particle operator $h_{\gamma,\lambda}(K,q_0),$ associated to the decomposition \eqref{decompos1e} of the channel operator. Since $h_{\gamma,\lambda}(K,\cdot)$ is an analytic family, there exists $k_0>0$ such that for any $q\in B_{2^{-k_0}}(q_0)$ the fiber  $h_{\gamma,\lambda}(K,q)$ has a unique eigenvalue $z(q)<\cE_{\gamma,\min}(K)\le \cE_{\gamma,\min}(K,q).$ Moreover, $q\mapsto z(q)$ is continuous in $B_{2^{-k_0}}(q_0)$ and we may assume that the associated normalized eigenfunctions  $q\mapsto f_{q}(p)\in L^2(\T^3)$ are also continuous in $B_{2^{-k_0}}(q_0).$ Clearly, the map $p\mapsto f_q(p)$ can be chosen real-analytic on $\T^3.$
As above set $U_k:=B_{2^{-k}}(q_0)\setminus \overline{B_{2^{-k-1}}(q_0)}$ for $k>k_0$ and define 
$$
\psi_k(p,q):=\frac{\chi_{U_k}(q)f_q(p)}{\sqrt{|U_k|}},\quad 
\phi_k(p,q) = \frac{\psi_k(p,q)-\psi_k(q,p)}{\sqrt{2 + \alpha_k}},
$$
where
$$
\alpha_k:= \frac{2}{|U_k|}\int_{U_k\times U_k} f_q(p)f_p(q)\,dpdq.
$$
By the H\"older inequality
$$
|\alpha_k| \le \frac{2}{|U_k|} \int_{U_k\times U_k} |f_q(p)|^2dpdq\to0
$$
as $(p,q)\mapsto f_q(p)$ is continuous in $B_{2^{-k_0}}(q_0)\times \T^3.$ Moreover, 
$$
\int_{\T^3\times\T^3} |\psi_k(p,q)|^2dpdq = \frac{1}{|U_k|}\int_{U_k} dq \int_{\T^3} |f_q(p)|^2dp = 1,
$$
we have $\|\phi_k\|_{L^2}=1.$ Furthermore, for any $g\in L^{2}((\T^3)^2)\cap C((\T^3)^2),$
\begin{align*}
\Big|\int_{\T^3\times\T^3} \phi_k(p,q)\cl{g(p,q)}\,dpdq \Big| \le & \frac{1}{\sqrt{2+\alpha_k}\sqrt{|U_k|}} \int_{U_k} dq \int_{\T^3} |f_q(p)| \,|g(p,q)-g(q,p)|\,dp \\
\le & \frac{1}{\sqrt{2+\alpha_k}\sqrt{|U_k|}} \int_{U_k} dq \sqrt{\int_{\T^3} |g(p,q)-g(q,p)|^2dp}\to0
\end{align*}
as $k\to+\infty,$ where in the last inequality we used the H\"older inequality together with assumption $\|f_q\|_{L^2}=1$ for all $q\in U_k.$ By approximation this inequality holds for all  $g\in L^{2}((\T^3)^2),$ and hence, $\phi_k\wk0$ in $L^2((\T^3)^2).$ Moreover, by the definition of the direct integral decomposition and the choice of $\psi_k$ and $f_q,$ for any $q,$
$$
H_{\gamma,\lambda}^{(2)}(K) \psi_k(\cdot,q) = 
\begin{cases}
0 & \text{$q\in \T^3\setminus U_k,$} \\
\frac{1}{\sqrt{|U_k|}}\,h_{\gamma,\lambda}(K,q)f_q & \text{$q\in U_k$}
\end{cases}
= 
\begin{cases}
0 & \text{$q\in \T^3\setminus U_k,$} \\
\frac{1}{\sqrt{|U_k|}}\,z(q)f_q & \text{$q\in U_k.$}
\end{cases}
$$
Thus, 
$$
\|(H_{\gamma,\lambda}^{(2)}(K) -z_0I)\psi_k\|_{L^2}^2 = \frac{1}{|U_k|}\int_{U_k\times \T^3} |z(q)-z_0|^2 |f_q(p)|^2dpdq = \frac{1}{|U_k|}\int_{U_k} |z(q)-z_0|^2dq \to0
$$
as $k\to+\infty.$
Moreover, as
$$
V_{23}\psi_k(p,q) = (2\pi)^{-3/2}\int_{\T^3}\psi_k(p,s)ds = (2\pi)^{-3/2}\,\frac{1}{\sqrt{|U_k|}}\int_{U_k}f_s(p)ds,
$$
using again the H\"older inequality
$$
\int_{\T^3\times\T^3}|V_{23}\psi_k(p,q) |^2dpdq \le \frac{1}{(2\pi)^3|U_k|} \int_{\T^3\times\T^3\times U_k} |U_k|\,|f_s(p)|^2dpdqds = |U_k|\to0
$$
as $k\to+\infty.$ Finally, as 
$$
\int_{\T^3\times\T^3} H_{\gamma,\lambda}^{(2)}(K) \psi_k(p,q)\psi_k(q,p)\,dpdq = \frac{1}{|U_k|} \int_{U_k\times U_k} z(q)f_p(q)\,f_q(p)dq,
$$
and hence, by the continuity of $(p,q)\mapsto  z(q)f_p(q)\,f_q(p)$ in $B_{2^{-k_0}}(q_0)\times \T^3$ it follows that the above integral also converges to $0$ as $k\to+\infty.$ The above discussions and the definition of $\phi_k$ imply
$$
\|(H_{\gamma,\lambda}(K) - z_0I)\phi_k\|_{L^2}^2 \to0
$$
as $k\to+\infty.$ Thus, again by the Weyl's criterion, $z_0\in \sigma_\ess(H_{\gamma,\lambda}(K)).$

\end{document}